\documentclass[12pt]{amsart}  
\usepackage[left=2.9cm, right=2.9cm, top=2cm, bottom=1.75cm]{geometry}  
\usepackage{amsmath}
\usepackage{amsfonts}
\usepackage[pdftex]{hyperref}
\usepackage{amsmath,amssymb}
\usepackage{amsthm}
\usepackage[usenames,dvipsnames]{color}
\usepackage{tikz}
\usepackage{accents}
\usepackage{xcolor}
\usepackage{hyperref}
\usepackage{mathtools}
\usepackage{float}
\usepackage{multicol}

\usepackage{listings}
\usepackage[T1]{fontenc}

\usepackage[font={small}]{caption}
\usetikzlibrary{decorations.pathmorphing}
\usepackage[normalem]{ulem}

\newcommand{\la}{\lambda}

\DeclareMathOperator{\spn}{span}

\DeclareMathOperator{\im}{Im}

\DeclareMathOperator{\Id}{Id}

\definecolor{bloesh}{RGB}{160, 194, 185}
\definecolor{roo}{RGB}{241,23,70}

\newfont{\gothic}{eufm10}

\newcommand{\R}{\mathbb{R}}
\newcommand{\Z}{\mathbb{Z}}
\newcommand{\N}{\mathbb{N}}

\newcommand{\1}{{\bf 1}}  
\newcommand{\0}{{\bf 0}}  
\newcommand{\lam}{\lambda}
\newcommand{\DeltaP}{{\Delta_\mathcal{P}}}
\newcommand{\calP}{\mathcal{P}}
\newcommand{\cupdot}{\mathbin{\mathaccent\cdot\cup}}

\theoremstyle{plain}
\newtheorem{thr}{Theorem}[section]

\newtheorem{lem}[thr]{Lemma}
\newtheorem{cor}[thr]{Corollary}
\newtheorem{prop}[thr]{Proposition}
\theoremstyle{definition}
\newtheorem{defi}[thr]{Definition}
\newtheorem{ex}[thr]{Example}
\theoremstyle{remark}
\newtheorem{remk}[thr]{Remark}
\theoremstyle{remark}

\definecolor{mycolor}{RGB}{190, 0, 190}

\begin{document}
\thanks{\today}

\title[Anti-Synchrony Subspaces of Weighted Networks]
{Invariant Synchrony and Anti-Synchrony Subspaces of Weighted Networks}

\author{Eddie Nijholt}
\address{Department of Mathematics, Imperial College London, United Kingdom }
\email{eddie.nijholt@gmail.com}

\author{N\'andor Sieben}
\author{James W. Swift}
\address{Department of Mathematics and Statistics, Northern Arizona University PO Box 5717, Flagstaff, AZ 86011-5717, USA}
\email{Nandor.Sieben@nau.edu,Jim.Swift@nau.edu}

\begin{abstract}
The internal state of a cell in a coupled cell network is often described by an element of a vector space. Synchrony or anti-synchrony occurs when some of the cells are in the same or the opposite state. Subspaces of the state space containing cells in synchrony or anti-synchrony are called polydiagonal subspaces. We study the properties of several types of polydiagonal subspaces of weighted coupled cell networks. In particular, we count the number of such subspaces and study when they are dynamically invariant. Of special interest are the evenly tagged anti-synchrony subspaces in which the number of cells in a certain state is equal to the number of cells in the opposite state. Our main theorem shows that the dynamically invariant polydiagonal subspaces determined by certain types of couplings are either synchrony subspaces or evenly tagged anti-synchrony subspaces. 
A special case of this result confirms a conjecture about difference-coupled graph network systems.   
\end{abstract}
\maketitle

\section{Introduction}
Networks are widely used as models throughout science and engineering. 
Examples include such diverse topics as flocking behavior in animals \cite{flock}, power grids, and neural connectomes \cite{coll}.
A central theme is the observation that complex networks have the remarkable capability of spontaneous self-organization. 
This is most often explored in the form of synchronization: 
the phenomenon where multiple cells in a network start displaying identical or highly related dynamical behavior. 
Synchronization has been linked to gene regulatory networks \cite{gene}, the workings of the brain \cite{Fries}, and opinion dynamics \cite{opinion}, among many other examples \cite{Arenas}.

Recently, another form of self-organization has received increasing attention, namely that of anti-synchronization. 
This occurs when two or more cells in the network show opposite behavior.
It has applications in the context of multi-agent systems \cite{hu2013bipartite}, chaotic oscillators \cite{KimAntiSynch, Liu2006}, Chua's oscillators \cite{miller1999synchronization}, optically coupled semiconductor lasers \cite{wedekind2001experimental}, and neural networks \cite{meng2007robust}.

In its most elementary form, synchrony means that two or more cells in the system have identical behavior.  
This can be modeled by assuming a flow-invariant subspace of the state space defined by one or more equalities of the form $x_i = x_j$, with $x_i$ describing the internal state of cell $i$. 
Similarly, anti-synchrony is often viewed as the presence of a flow-invariant subspace determined by the aforementioned conditions, as well as those of the form $x_i = -x_j$ or $x_k = 0$. 
We will refer to any of these subspaces as polydiagonal subspaces. 
If only equalities of the form $x_i = x_j$ are needed to describe the subspace, then we will speak of a synchrony subspace. 
In all other cases, when at least one relation of the form $x_i = -x_j$ or $x_k = 0$ is in effect, we will speak of an anti-synchrony subspace.

To account for models with uncertain or tunable details,  
one is often interested in polydiagonal subspaces that are invariant for entire families of network ODEs at once.
The structure of a coupled cell network is 
modeled by a weighted digraph.
We are interested in classifying the subspaces that are invariant for all systems 
with a given digraph \cite{AD2020, NSS6, Pivato}.
These are often referred to as robust invariant subspaces.
Surprisingly, there are many results showing that robust polydiagonal subspaces,
while required to be invariant for all network ODEs,
can be characterized as those that are robust for all linear network ODEs. This characterization allows one to use results from linear algebra to classify robust polydiagonal subspaces or to identify rules of implication or exclusion
\cite{AD2020, Aguiar&Dias, Torok, NSS6}.

For instance, Golubitsky, Stewart and collaborators have developed the so-called groupoid formalism to define admissible vector fields, which may be seen as the most general class of ODEs that are compatible with a given network structure \cite{Pivato}.
It is later shown that a synchrony space is invariant under the flow of all admissible vector fields, if and only it is invariant for all \emph{linear} admissible vector fields \cite{Torok, Aguiar&Dias}.

In \cite{NSS6}, the authors consider instead so-called difference-coupled vector fields with respect to a graph, which have the component-wise description 
\begin{align}\label{diff-coupled}
f_i(x) = g(x_i) + \sum_{j \in N(i)}h(x_j - x_i)\, .
\end{align} 
Here each component corresponds to a node $i$ in the graph, which is given a variable $x_i$ in some  phase space $V$ that is the same for each of the $n$ nodes, and $N(i)$ denotes the neighbors of $i$. 
Correspondingly, the full ODE is defined on the total phase space $V^n$, with $x = (x_1, \dots, x_n) \in V^n$ denoting the combined variable of the system. 
The functions $g, h: V \to V$ may in general be non-linear.

Multiple classes of difference-coupled vector fields are considered, by imposing various natural restrictions on $f$ and $g$ in Equation \eqref{diff-coupled}. 
The authors then classify those polydiagonal subspaces that are invariant for all vector fields in a given class.  
Again, in many cases the result is that these polydiagonal subspaces are precisely those that are invariant under a given linear map.

These results are generalized in \cite{AD2020} by considering systems of the form 
\begin{align}\label{weight-additive-coupled}
f_i(x) = g(x_i) + \sum_{j =1}^n w_{ij}h(x_i, x_j)\, .
\end{align} 
Here the coupling strength $w_{ij} \in \R$ is the weight of the directed edge from cell $j$ to cell $i$, and $h: V \times V \to V$ now takes two inputs. 
As in \cite{NSS6}, different classes of vector fields are considered by imposing constraints on $f$ and $g$, and invariant polydiagonal subspaces are considered for each of these classes. 
In many cases a polydiagonal subspace is invariant under all vector fields in a given class, if and only if it is invariant under the weight matrix $W = (w_{ij})$ or the Laplacian matrix associated to it.

For instance, the authors of \cite{AD2020} consider all systems of the form Equation \eqref{weight-additive-coupled} with $g$ odd, and $h$ even in its first argument and odd in its second (the so-called even-odd-input-additive coupled cell systems). 
They then show that a polydiagonal subspace is invariant under the flow of all such systems, if and only if it is kept invariant by the weight matrix $W$, where we assume $V = \R$ for convenience.

These results illustrate the importance of determining polydiagonal subspaces that are invariant under certain linear maps, as in many cases this equates to invariance under a broad class of generally non-linear systems. Guided by the importance of linear network ODEs, we study a class of coupled cell networks with identical cells and linear coupling.
Two examples, coupled van der Pol oscillators, and coupled Lorenz attractors, are used to illustrate results about dynamically invariant polydiagonal subspaces.
The anti-synchrony of coupled Lorenz attractors observed in \cite{KimAntiSynch} is explained
by our results.
We also demonstrate a new mechanism for anti-synchrony in this system.

A major motivation for this paper is \cite[Conjecture 5.3]{NSS6}. It states that in an anti-synchrony subspace that is invariant under the graph Laplacian of an unweighted graph, the number of cells in a common state must be the same as the number of cells in the opposite state. 
This allows, for example, the subspaces with 
\emph{typical elements} of the form $(a,-a,b,-b,0)$ or $(a,-a,a,-a)$ but not of the form $(a,a,-a)$.
This conjecture is proved here as Theorem~\ref{conjecture53} as 
a corollary of our main theorems.
One interesting consequence is that such networks with an odd number of cells always have at least one node with vanishing dynamics in any 
anti-synchrony subspace.
We also point out that the aforementioned condition rules out 
the vast majority of the
anti-synchrony subspaces.
This can make it significantly easier to find all such subspaces in a systematic way.
 
To explore the different implications on polydiagonal subspaces imposed by a linear map,  we introduce various types of polydiagonal subspaces.
An anti-synchrony subspace satisfying the conclusion of Theorem~\ref{conjecture53}
is called 
\emph{evenly tagged}. 
Slightly less restrictive is the concept of a  
\emph{fully tagged anti-synchrony subspace}. 
Here every node has at least one node with opposite dynamics. 
If this is the same node, then it necessarily has vanishing dynamics. 
Examples are given by the aforementioned evenly tagged subspaces, as well as those with typical element $(a,a,-a)$. 
At the other extreme within the set of anti-synchrony subspaces sit the 
\emph{minimally tagged subspaces}. 
Here we do not have any relations of the form $x_i = -x_j$, unless we have 
$i = j$ which implies $x_i = 0$.
Examples include those with typical elements $(a,a,0,0)$ and $(a,b,a,0,c,a)$. 
Note that the only anti-synchrony subspace that is both minimally tagged and fully tagged is the 
trivial subspace $\{ {\bf {0}}\}$. 

This paper is organized as follows. In 
Section~\ref{Preliminaries} we introduce the basics of weighted digraphs and linearly coupled cell networks used in subsequent sections. 
Section~\ref{TaggedPartitionsandPolydiagonalSubspaces} then focuses on the different types of polydiagonal subspaces and their relation to so-called tagged partitions of sets. 
Next, Section~\ref{Dynamicallyinvariantsubspaces} investigates dynamically invariant
subspaces for linearly coupled cell networks.
Having thus established our main motivation for looking at polydiagonal subspaces, we next look at structural implications of connection matrices on corresponding subspaces.  
In Section \ref{ALemmaaboutMInvariantSubspaces} we formulate and prove the main technical result of this paper in the form of Lemma \ref{MainLemma}. 
It shows how certain properties of a matrix have important implications on its invariant subspaces.  
Section \ref{Results_for_Matrices_with_Constant_Column_Sums} further builds on this result, 
by showing how the property of having constant column sums has strong implications on the anti-synchrony subspaces of a matrix. 
Important examples of matrices with this property include the Laplacian of a graph and the adjacency matrix of a weighted digraph obtained from the Cayley color digraph of a group. 
Here we also prove and generalize \cite[Conjecture 5.3]{NSS6}. 
Finally, in Section \ref{Counting_polydiagonal_subspaces} we investigate the number of polydiagonal subspaces in $\R^n$ belonging to each of our classes. 
We find that some of the relevant number sequences establish surprising connections to other areas of mathematics, whereas other sequences seem to have been previously unknown.  

\section{Preliminaries}\label{Preliminaries}

\subsection{Weighted digraphs}

A \emph{digraph} $G$ is a pair $(V_G,E_G)$, where $V$ is a finite set of vertices and $E_G\subseteq V\times V$ is the set of arrows. 
We usually let $V=\{1,\ldots,n\}$. 
The arrow $(i,j) \in E_G$ connects the tail $i$ to the head $j$.
Note that our digraphs can have loops and do not have multiple arrows from one vertex to another. A \emph{weighted digraph} is a pair $(G,w)$, where $G$ is a digraph and $w:E_G\to \mathbb{R}\setminus \{0\}$ is a \emph{weight function}. The \emph{underlying digraph} of the weighted digraph $(G,w)$ is the digraph $G$. Every digraph can be treated as a weighted digraph with a default weight function whose value is $1$ for every edge.

Consider a weighted digraph $(G,w)$ with vertex set $\{1,\ldots, n\}$. The \emph{adjacency matrix} of the weighted digraph is the $n\times n$ matrix $A$ defined by 
$$
A_{i,j}:=
\begin{cases}
w(j,i) & \text{if } (j,i)\in E_G \\
0      & \text{otherwise}.
\end{cases}
$$ 
Note that $A_{i, j}$ is the weight of the 
arrow to vertex $i$ from vertex $j$.
So our adjacency matrix is an in-adjacency matrix, which is the standard for coupled cell networks \cite{AD2020}.
The adjacency matrix provides a bijection from the set of weighted digraphs with vertex set $\{1,\ldots , n\}$ to $\mathbb{R}^{n\times n}$. 
The \emph{Laplacian matrix} of the weighted digraph is the $n\times n$ matrix $L$ defined by 
\begin{equation}
\label{Laplacian}
L_{i,j} := \left (\sum_{k = 1}^n  A_{i,k} \right )  \delta_{i,j}  - A_{i,j}=
\begin{cases}
 -A_{i,j} & \text{if }i \neq j\\
 \sum_{k \neq i} A_{i, k} & \text{if }i = j ,
\end{cases}
\end{equation}
where $\delta$ is the Kronecker delta function. 

For a vertex $i$ of a weighted digraph $(G,w)$, the \emph{in-neighborhood} of $i$ is $N^-(i):=\{j\mid (j,i)\in E_G\}$ and the \emph{out-neighborhood} of $i$ is $N^+(i):=\{j\mid (i,j)\in E_G\}$. The \emph{in-degree} and \emph{out-degree} of $i$ are 
\[
d^-(i):=\sum_{j\in N^-(i)} w(j,i), \qquad d^+(i):=\sum_{j\in N^+(i)} w(i,j).
\] 
These values can be computed as the row and column sums  
\[
d^-(i)=\sum_{j=1}^n A_{i,j}, \qquad d^+(j)=\sum_{i = 1}^n A_{i,j}
\]
of the adjacency matrix $A$.

Generalizing \cite{imbalance}, 
the \emph{imbalance} of $i$ is $b(i):=d^+(i)-d^-(i)$. 
The weighted digraph is called \emph{weight-balanced} if the imbalance of every vertex is 0. To avoid cluttered terminology, we will often refer to a weight-balanced weighted digraph as simply a weight-balanced digraph. 
For connected unweighted digraphs, the notion of weight-balanced (using the default weight function) is equivalent to the existence of an Euler circuit. 
Note that a weighted digraph is weight-balanced if and only if the $i$-th row sum of the adjacency matrix is equal to the $i$-th column sum for all $i$.

A \emph{graph} is a pair $(V,E)$, where $V$ is a finite set of vertices and $E$ is the set of edges, where each edge is a set of two vertices.

The \emph{underlying graph} of a digraph $G$ is the graph with vertex set $V = V_G$ and edge set $E = \{\{i,j\}\mid (i,j)\in E_G\}$. A digraph is \emph{weakly connected} if its underlying graph is connected. A digraph is \emph{strongly connected} if there is a directed path from any vertex to any other vertex. A weighted digraph is weakly (strongly) connected if the underlying digraph is weakly (strongly) connected. A square matrix is called \emph{irreducible} if it is the adjacency matrix of a strongly connected weighted digraph.

Given a graph, the \emph{corresponding digraph} is built from the graph by replacing every edge by a pair of arrows. An edge $\{i,j\}$ is replaced by the arrows $(i,j)$ and $(j,i)$. The \emph{Laplacian matrix} of a graph is the Laplacian matrix of its corresponding digraph.

An \emph{automorphism of a weighted digraph} 
$(G,w)$ 
is a bijection $\phi:V_G\to V_G$ such that $(\phi(i),\phi(j))\in E_G$ if and only if $(i,j)\in E_G$, and $w(\phi(i),\phi(j))=w(i,j)$ for all $(i,j)\in E_G$. The set of automorphisms under composition is the \emph{automorphism group of the weighted digraph} $(G,w)$.
In other words, the automorphism group of a weighted digraph is the subgroup of the automorphism group of the underlying digraph consisting of the weight preserving automorphisms. 

\subsection{Coupled cell networks}

\label{CoupledCellNetworks}
Polydiagonal subspaces are of interest in the study of coupled cell networks.  The dynamics of a coupled cell network is
modeled by a system of
ordinary differential equations (ODEs) with \emph{cells} in $C := \{1, \ldots, n\}$
with a state variable $x_i \in \R^k$ for $i \in C$.  
The state of the system is $x = (x_1, \ldots, x_n) \in (\R^k)^n$.
We consider the case where each cell
has identical  \emph{internal dynamics} governed by ${\dot x}_i = f(x_i)$ for
a given function $f: \R^k \to \R^k$.  We assume that the function $f$ is smooth so
that the initial value problem has a unique solution.
A coupled cell network is represented by a weighted digraph $(G,w)$ with vertices for the cells and arrows corresponding to the coupling.

In this paper, we restrict our attention to linearly coupled cell networks defined by
the internal dynamics, 
the adjacency matrix $A \in \R^{n \times  n}$ for $(G,w)$,
and a \emph{coupling matrix} $H \in \R^{k\times k}$.
The system of ODEs is
\begin{equation}
\label{ODEwithA}
{\dot x}_i =  f(x_i) + \sum_{j \in N^-(i)} w(j,i) H x_j = f(x_i) + \sum_{j = 1}^n A_{i,j} H x_j.
\end{equation}
This system is a special case of \cite[Equation (1.1)]{AD2020}, with 
$H x_j = h(x_i, x_j)$.
Another common system, especially in the physics literature, is
\begin{equation}
\label{ODEwithL}
{\dot x}_i = f(x_i) + \sum_{j=1}^n A_{i,j} H(x_i - x_j) = f(x_i) + \sum_{j=1}^n L_{i,j}  H x_j.
\end{equation}
This is often called diffusive coupling, since 
$L$ is the Laplacian matrix of the weighted digraph, defined in Equation~(\ref{Laplacian}).
See for example \cite[Equation(1)]{Diggans21}, where $\hat H$ is used for our $H$.

Note that the Laplacian matrix of a weighted digraph $G$ is the adjacency matrix for a different weighted digraph $G_L$, so the distinction between
$A$ and $L$ is somewhat arbitrary for linear coupling.
That being said, the row sums of a Laplacian matrix are all zero, 
and consequently $\1:=(1,1,\ldots,1)$ is an eigenvector of $L$ with  eigenvalue 0.

Note that Systems (\ref{ODEwithA}) and (\ref{ODEwithL}) can each be written as
$\dot x = F(x)$, for a function $F: (\R^k)^n \to (\R^k)^n$ defined by
\begin{equation}
\label{ODE}
{\dot x}_i = F_i(x) := f(x_i) + H \sum_{j=1}^n M_{i,j} x_j.
\end{equation}
The adjacency matrix $M$ is arbitrary, and we typically use $A$ as a general adjacency matrix, or  $L$ 
if the row sums are all $0$.  The matrix $H$ has been moved to multiply the sum rather than each term.

\begin{ex}
\label{van_der_Pol}
A system of $n$ linearly coupled van der Pol oscillators is described by
\begin{equation}
\label{coupled_vdp}
{\ddot u}_i = -\varepsilon (1-u_i^2) {\dot u}_i - u_i + (M u)_i
\end{equation}
for a function $u: \R \to \R^n$.
This can be put in the form of System~(\ref{ODE}) with $k = 2$.  
The internal dynamics and coupling matrix are $f(u,v) = (v, -\varepsilon(1-u^2)v - u)$ and
$H = \left [ \begin{smallmatrix}
0 & 0 \\
1 & 0
\end{smallmatrix}
\right ]
$.
\end{ex}

\begin{ex}
\label{Lorenz}
A system of $n$ coupled Lorenz attractors can be written as 
System~(\ref{ODE}) with $k = 3$ and internal dynamics 
$f(u,v,w) = (10(v-u), u(28 - w) - v, u v - \frac8 3 w)$.
These systems have attracted significant interest, see for example \cite{PecoraCarroll90, KimAntiSynch, Diggans21}.
The 9 possible $H$ matrices with exactly one non-zero entry are considered in \cite{PhysRevE_Pecora09}.
\end{ex}

\section{Tagged Partitions and Polydiagonal Subspaces}
\label{TaggedPartitionsandPolydiagonalSubspaces}

Consider a cell network in which the internal state of every cell can be described by an element of a vector space. Some of these cells can be in the same state, while others can be in the opposite state. To describe this situation, we want to group all the cells that are in the same state into a single class. These classes partition the set of cells. We also want to indicate the pairs of classes that are in the opposite state. This requires a mapping of some of the classes to other classes. These goals motivate the use of the following mathematical tools.

A \emph{partial function} between the sets $X$ and $Y$ is a function
$f:\tilde{X}\to Y$, where $\tilde{X}$ is a subset of $X$. We do
allow $\tilde{X}$ to be the empty set in which case the partial function
is the empty function. We also allow $\tilde{X}=X$ in which case
we say that the partial function is \emph{fully defined}. A \emph{partial
involution} on $X$ is a partial function $x\mapsto x^{*}: \tilde X \to \tilde X$ on $X$ that is its own inverse.

\begin{defi}
\label{tagged}
A \emph{tagged partition} of a finite set $C$ is a partition $\mathcal{P}$
of $C$ together with a partial involution 
$P\mapsto P^{*}: \tilde{\mathcal{P}} \to \tilde{\mathcal{P}}$ on $\mathcal{P}$
that has at most one fixed point. 
\end{defi}

Note that $P^{**}=P$ follows from the definition of a partial involution.  We denote the fixed point of the partial involution, if it exists, by $P_0 = P_0^*$.

Recall that a partition $\mathcal{P}$ of a set $C$ determines an equivalence relation on $C$. This relation is sometimes denoted by $\bowtie$ in the literature. The equivalence class of $i\in C$ in this relation is denoted by $[i]$.

\begin{defi}
If $\mathcal{P}$ is a tagged partition of $C=\{1,\ldots,n\}$, then
\[
\Delta_{\mathcal{P}}:=\{x\in\mathbb{R}^{n}\mid x_{i}=x_{j}\text{ if }[i]=[j]\text{ and }x_{i}=-x_{j}\text{ if }[i]^{*}=[j]\}
\]
is the \emph{polydiagonal subspace for} $\mathcal{P}$. 
A subspace $\Delta$ of $\R^n$ is called a \emph{polydiagonal subspace} if $\Delta = \Delta_{\mathcal{P}}$ for some tagged partition $\mathcal{P}$.
\end{defi}

Note that $x_{i}=0$ if $x\in\Delta_{\mathcal{P}}$ and $[i]^{*}=[i]$ in $\mathcal{P}$. Also note that every polydiagonal subspace
corresponds to a unique tagged partition. The uniqueness is guaranteed because the partial involution has at most one fixed point.
So $\mathcal{P}\mapsto\Delta_{\mathcal{P}}$ is a bijection between the tagged partitions of $C$ and the polydiagonal subspaces of $\R^n$.

\begin{ex}
\label{ex:polydiagonal}
The polydiagonal subspace of the tagged partition 
$$
\mathcal{P}=\{\{1\},\{2,4\},\{3\},\{5,6\}\}
$$ 
with $\{1\}^*=\{2,4\}$ and $\{5,6\}^*=\{5,6\}$ is $\Delta_\mathcal{P}=\{(a,-a,b,-a,0,0) \mid a,b\in \mathbb{R} \}$. We say that $(a,-a,b,-a,0,0)$ is the \emph{typical element} of $\Delta_\mathcal{P}$. 

This subspace can also be written as $\{x \in \R^6 \mid x_1 = - x_2 = - x_4, \ x_5 = x_6 = 0\}$ or as the null space of 
\[
\left[
\begin{smallmatrix}
1 & 0 & 0 & 1  & 0 & 0\\ 
0 & 1 & 0 & -1 & 0 & 0\\ 
0 & 0 & 0 & 0  & 1 & 0\\ 
0 & 0 & 0 & 0  & 0 & 1\\ 
\end{smallmatrix}
\right]
\]
or as $\spn\{(1,-1,0,-1,0,0),(0,0,1,0,0,0)\}$. Each of these representations have advantages and disadvantages with respect to categorization, information density, simplicity, ease of use, uniqueness, and computational efficiency.
\end{ex}

Although a tagged partition is a simple and well-behaved mathematical object that efficiently and uniquely encodes a polydiagonal subspace, the notation for describing specific examples of tagged partitions can be cumbersome. As the previous example shows, the typical element can provide a concise notation for a concrete polydiagonal subspace. 

We are interested in several special types of polydiagonal subspaces.

\begin{defi}
\label{TypesOfPolydiagonals}
Let $\mathcal{P}$ be a tagged partition of $C=\{1,\ldots,n\}$. We say that the polydiagonal subspace $\Delta_{\mathcal{P}}$ is a
\begin{enumerate}
\item \emph{synchrony subspace} if the partial involution is the empty function;
\item \emph{anti-synchrony subspace} if the partial involution is 
not the empty function.
\end{enumerate}
We say that $\DeltaP$ is
\begin{itemize}
\item \emph{minimally tagged} if the domain of the partial involution is a singleton set;
\item \emph{fully tagged} if the partial involution
is fully defined;
\item \emph{evenly tagged} if it is fully tagged and
$\#P=\#P^{*}$ for all $P\in\mathcal{P}$.
\end{itemize}
\end{defi}

\begin{remk}
The following table shows  examples of typical elements of different types of polydiagonal subspaces.
We follow a convention that the first $a$ precedes both $-a$ and $b$, and the first $b$ precedes both $-b$ and $c$, etc.
\begin{center}
\begin{tabular}{ll}
\hline
Synchrony subspace & $(a, b, a, c)$\\
Anti-synchrony subspaces \\
$\quad$ minimally tagged & $(a, b, 0)$\\
$\quad$ fully tagged & $(a, a, -a, 0, b, -b)$\\
$\quad$ evenly tagged & $(a, a, -a, 0, b, -b, -a)$ \\
\hline
\end{tabular}
\end{center}
The typical element of a synchrony subspace has no $0$'s and no negative signs, and the
typical element of an anti-synchrony subspace has at least one $0$ or at least one negative sign.

The typical element of a minimally tagged polydiagonal subspace has at least one $0$ component, and no negative signs.  
The typical element of a fully tagged polydiagonal subspace has a component equal to $-a$ whenever there is a component equal to $a$, and the same for $b$, $c$, etc.  There can be any number of $0$s, including all $0$s or no $0$s at all.  
Finally, the typical element of an evenly tagged polydiagonal subspace has the same number of $-a$'s as there are $a$'s, and the same for $b$, $c$, etc.  
\end{remk}

\begin{remk}
Definition~\ref{tagged} of a tagged partition is equivalent to \cite[Definition 3.1]{AD2020}, which is given in terms of integers $p$, $q$ and $r$ with $0 \leq q \leq p$ and $r \in \{0, 1\}$.
 Their tagged partition
\[
\mathcal{P}=\{P_k \mid 1 \leq k \leq p\} \cup \{\overline{P_l} \mid 1 \leq l \leq q\} \cup \{ P_m \mid 0 \le m \le r-1 \} 
\]
contains $p+q+r$ classes. The connection with our definition is that $P_l^*  = \overline{P_l}$ and $\overline{P_l}^{\,*}  =  P_l$ for $ 1 \leq l \leq q$, $P_k \not \in \tilde \calP$ for $k > q$, and $P_0^* = P_0$ if $r = 1$.

Definition~\ref{TypesOfPolydiagonals} can be written in terms of the notation in \cite{AD2020} as follows:
A synchrony subspace has $q = 0$ and $r = 0$, an anti-synchrony subspace has $q>0$ or $r = 1$, a minimally tagged subspace has $q = 0$ and $r = 1$, 
a fully tagged subspace has $p = q$, and an evenly tagged subspace has $p = q$ and $\# P_k = \# P_k^*$ for all $k$.
\end{remk}

\begin{remk}
The reader familiar with \cite{AD2020, Pivato} will
know that they use the word ``polydiagonal'' differently.  
We have chosen our convention to better distinguish the different types of subspaces that will play a role throughout this paper. 
We summarize the differences below:
\begin{center}
\begin{tabular}{ll}
our terminology & terminology of \cite{AD2020}\\
\hline
polydiagonal subspace & generalized polydiagonal subspace\\
synchrony subspace & polydiagonal subspace \\
anti-synchrony subspace & generalized polydiagonal subspace for a \\
& non-standard tagged partition \\
\end{tabular}
\end{center}
\end{remk}

\begin{figure}
\includegraphics{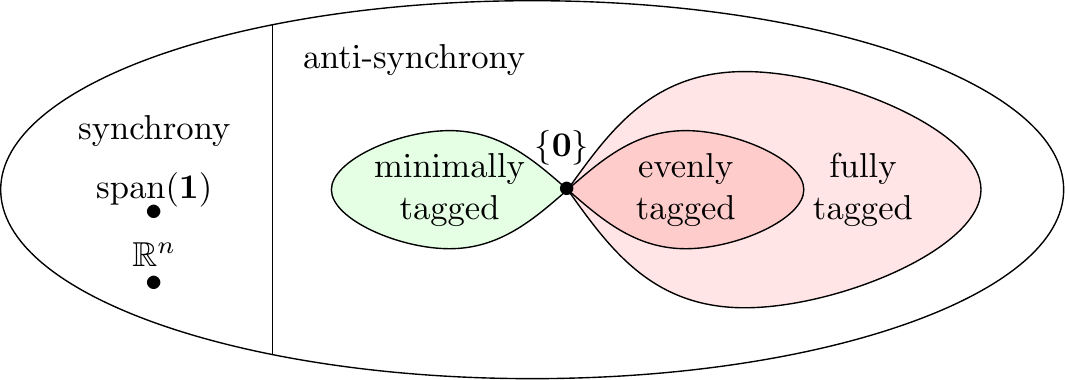}
\caption{
\label{classification}
Classification of the polydiagonal subspaces of 
$\R^n$, with $n \in \N$.
}
\end{figure}

\begin{ex}
For $n \in \N$, $\R^n$ and $\spn(\1)$ are synchrony subspaces, 
and the trivial subspace $\{\0\}$ is an evenly and minimally tagged anti-synchrony subspace of $\R^n$.
The fully synchronous subspace $\spn(\1)$ plays an important role in the theory of coupled cell networks.
The trivial subspace is the only anti-synchrony subspace that is both minimally and fully tagged.  
The subsets of polydiagonal subspaces are shown in Figure~\ref{classification}.
\end{ex}

\begin{ex}
\label{R0}
The only polydiagonal subspace of $\mathbb{R}^0=\{()\}$ is $\mathbb{R}^0=\Delta_\mathcal{P}$, where $\mathcal{P}=\emptyset$ is the empty partition of $C=\emptyset$. Hence the partial involution is the empty function. So $\mathbb{R}^0$ is an evenly, and fully tagged synchrony subspace but not an anti-synchrony subspace and not minimally tagged.
\end{ex}

Note that the set of polydiagonal subspaces is partitioned into synchrony subspaces and anti-synchrony subspaces.

The set of polydiagonal subspaces of $\R^n$ ordered by reverse inclusion is a lattice,
as shown in  \cite{Aguiar&Dias,AD2020}. 
Recall that a lattice is a partially ordered set in which every two elements have a unique least upper bound and a unique greatest lower bound.
Code to compute the lattice of polydiagonal subspaces is provided at \cite{ourGitHub}.

\begin{figure}
\begin{center}
\begin{tabular}{llll}
$i$ & $\mathcal{P}_i$ &$\Delta_{\mathcal{P}_i}$ & type of subspace  \tabularnewline
\hline
0 & $\{\{1,2\}^*=\{1,2\}\}$ &$\{(0,0)\}$  & trivial \tabularnewline
1 & $\{\{1\}, \{1\}^*=\{2\}\}$ &$\{(a,-a)\mid a\in\mathbb{R}\}$ & evenly tagged anti-synchrony \tabularnewline
2 & $\{\{1\}^*=\{1\},\{2\}\}$ &$\{(0,a)\mid a\in\mathbb{R}\}$ & minimally tagged anti-synchrony \tabularnewline
3 & $\{\{1\},\{2\}^*=\{2\}\}$ &$\{(a,0)\mid a\in\mathbb{R}\}$ & minimally tagged anti-synchrony \tabularnewline
4 & $\{\{1,2\}\}$     &$\{(a,a)\mid a\in\mathbb{R}\}$ & synchrony \tabularnewline
5 & $\{\{1\},\{2\}\}$ &$\{(a,b)\mid a,b\in\mathbb{R}\}$ & synchrony \tabularnewline
\hline
\end{tabular}

\hfil
\begin{tikzpicture}[xscale=.9]
\node (5) at (0,0) {$\Delta_{\mathcal{P}_5}$};
\node (1) at (-1.5,1) {$\Delta_{\mathcal{P}_1}$};
\node (3) at (-0.5,1) {$\Delta_{\mathcal{P}_2}$};
\node (2) at (0.5,1) {$\Delta_{\mathcal{P}_3}$};
\node (4) at (1.5,1) {$\Delta_{\mathcal{P}_4}$};
\node (0) at (0,2) {$\Delta_{\mathcal{P}_0}$};
\draw (0)--(1);
\draw (0)--(2);
\draw (0)--(3);
\draw (0)--(4);
\draw (1)--(5);
\draw (2)--(5);
\draw (3)--(5);
\draw (4)--(5);
\end{tikzpicture}
\hfil
\includegraphics{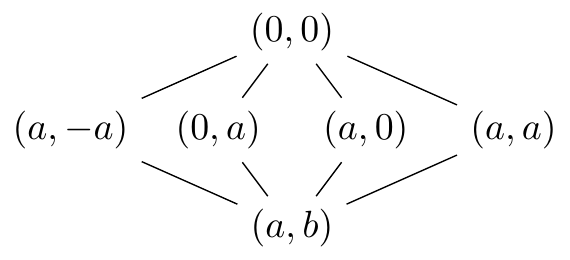}
\end{center}

\caption{\label{fig:ntwo}
The polydiagonal subspaces of $\R^2$. 
Subspaces are indicated by their typical elements
in the second Hasse diagram of the lattice. 
}
\end{figure}

\begin{ex}
\label{polyR2}
There are six polydiagonal subspaces of $\mathbb{R}^{2}$, as shown with their characterizations in Figure~\ref{fig:ntwo}. The lattice of these subspaces is also shown.
\end{ex}

\begin{ex}
\label{PolyR3} 
There are 24 polydiagonal subspaces of $\mathbb{R}^{3}$, divided into
5 synchrony subspaces and 19 anti-synchrony subspaces as shown in Figure~\ref{PolyR3Fig}.

\begin{figure}
$$
\begin{array}{ll}
\hline
\text{5 synchrony: } & (a,a,a), (a,b,b), (a,b,a), (a, a, b), (a,b,c)  \\
\text{9 nontrivial minimally tagged:}  &  (a, 0,0), (0, a, 0), (0, 0, a)  \\
& (0,a,a), (a, 0, a), (a, a, 0)  \\
& (0,a,b), (a, 0, b), (a, b, 0)  \\
\text{4 evenly tagged: } &   (0,0,0), (0, a, -a),(a, 0, -a), (a, -a, 0)  \\
\text{3 fully but not evenly tagged: } &  (a, -a, -a),(a, -a, a), (a, a, -a)  \\
\text{3 other anti-synchrony: } &  (a,b,-b), (a, b, -a), (a, -a, b) \\
\hline
\end{array}
$$
\caption{
\label{PolyR3Fig}
Typical elements of the 24 polydiagonal subspaces of $\R^3$.}
\end{figure}
\end{ex}

\begin{prop}
\label{orthodiagonal}
A polydiagonal subspace $\Delta_\mathcal{P}$ is an evenly tagged anti-synchrony subspace if and only if $\1 \in \Delta_\mathcal{P}^\perp$.
\end{prop}

\begin{proof} 
The forward direction of the statement 
follows from the observation that  $\1 \cdot x = \sum x_i$. 
Now assume $\1 \in \Delta_\mathcal{P}^\perp$. First, suppose that $\Delta_\mathcal{P}$ is not fully tagged. Then $[k]\in\mathcal{P}$ is untagged for some $k$. Define 
$$
x_i:=\begin{cases}
1 & \text{if } i\in [k] \\
0 & \text{otherwise}. 
\end{cases}
$$
Then $x\in\Delta_\mathcal{P}$ but $\1\cdot x=\#[k] \neq 0$, which is a contradiction.
Next, suppose that $\Delta_\mathcal{P}$ is fully but not evenly tagged. Then $\#[k]\neq\#[l]$ for some $[k],[l]\in\mathcal{P}$ satisfying $[k]^*=[l]$. Define 
$$
x_i:=\begin{cases}
1 &  \text{if } i\in [k] \\
-1&  \text{if } i\in [l] \\
0 & \text{otherwise}. 
\end{cases}
$$
Then $x\in\Delta_\mathcal{P}$ but $\1\cdot x=\#[k]-\#[l] \neq 0$, which is a contradiction.
\end{proof}

\begin{remk}
Proposition \ref{orthodiagonal} shows that evenly tagged anti-synchrony subspaces are characterized as polydiagonal subspaces that are \emph{orthodiagonal}, meaning they are orthogonal to $\1$. These orthodiagonal subspaces are the main subject of this paper.
\end{remk}

\section{Dynamically Invariant Subspaces}\label{Dynamicallyinvariantsubspaces}

A linear subspace $W \subseteq (\R^k)^n$ is 
\emph{dynamically invariant for System} (\ref{ODE}), 
or simply \emph{dynamically invariant},
provided that  
the solution to the ODE with any initial condition
in $W$ stays in $W$ on the interval of existence of that solution.

Let $\mathcal{P}$ be a tagged partition of $C=\{1,\ldots,n\}$.  Recall that the \emph{tensor product} of $\R^k$ with $\Delta_{\mathcal P}$ is
$$
\R^k \otimes \Delta_{\mathcal P} := \{ x \in (\R^k)^n \mid x_i = x_j \text{ if } [i] = [j] \text{ and } x_i = - x_j \text{ if } [i]^* = [j]  
\} .
$$

Note that if $\Delta_\mathcal{P}$ is an $M$-invariant polydiagonal subspace of $\R^n$, then 
\begin{equation}
\label{Minv}
\left ( H \sum_{j=1}^n M_{1,j} x_j,  \ldots, H \sum_{j=1}^n M_{n,j} x_j \right ) \in \R^k \otimes \Delta_\mathcal{P}
\end{equation}
for all $(x_1, \ldots, x_n) \in \R^k \otimes \Delta_\mathcal{P}$.
It follows that if $\Delta_\mathcal{P}$ is an $M$-invariant synchrony subspace of $\R^n$, then $\R^k \otimes \Delta_\mathcal{P}$ is dynamically invariant for any choice of 
$f: \R^k\to \R^k$. 
These dynamically invariant subspaces are defined by equations of the form 
$x_i = x_j$.

However, $\R^k \otimes \Delta_\mathcal{P}$ is not dynamically invariant for System (\ref{ODE}) for all $f$ if $\Delta_\mathcal{P}$ is an $M$-invariant anti-synchrony subspace of $\R^n$.
There are two restrictions to the functions $f$ that are important given this fact. 

If $f: \R^k\to \R^k$ satisfies $f(0) = 0$ and 
$\Delta_\mathcal{P}$ is an $M$-invariant synchrony subspace or 
minimally tagged anti-synchrony subspace, then
$\R^k \otimes \Delta_\mathcal{P}$ is dynamically invariant.
These dynamically invariant subspaces are defined by equations of the form 
$x_i = x_j$ and $x_i = 0$, but not $x_i = -x_j$ for $i \neq j$.  

A stronger restriction is that $f$ is \emph{odd}, meaning that $f(-x_i) = - f(x_i)$ for all $x_i \in \R^k$.  If $f$ is odd and $\Delta_\mathcal{P}$ is $M$-invariant, then $\R^k \otimes \Delta_\mathcal{P}$ is dynamically invariant.

The last few paragraphs have described the obvious consequences of Equation~(\ref{Minv}).
The next theorem says that the converse statements are also true.

\begin{prop}
\label{dynamicallyInvariant}
Let $W$ be a subspace of $(\R^k)^n$, and consider the dynamics of 
System~$(\ref{ODE})$ for fixed $M \in \R^{n \times n}$ and fixed nonzero $H \in \R^{k \times k}$.
\begin{enumerate}
\item $W$ is dynamically invariant for all odd $f$ if and only if 
$W = \R^k \otimes \Delta_\mathcal{P}$ for some $M$-invariant polydiagonal
subspace  
$\Delta_\mathcal{P}$.  
\item $W$ is dynamically invariant for all $f$ that satisfy $f(0) = 0$ if and only if 
$W = \R^k \otimes \Delta_\mathcal{P}$ for some $M$-invariant 
synchrony subspace or minimally tagged anti-synchrony subspace
$\Delta_\mathcal{P}$.  
\item $W$ is dynamically invariant for all $f$ if and only if 
$W = \R^k \otimes \Delta_\mathcal{P}$ for some $M$-invariant synchrony
subspace  
$\Delta_\mathcal{P}$.  
\end{enumerate}
\end{prop}

\begin{proof}
In all three cases, if $\Delta_\mathcal{P}$ is a polydiagonal
subspace satisfying the given conditions (e.g. synchrony or minimally tagged anti-synchrony) then it is clear
that $\R^k \otimes \Delta_\mathcal{P}$ is invariant under
the dynamics of System~(\ref{ODE}) with the given class of $f$.

Conversely, assume first that $W$ is invariant under
the dynamics of System~(\ref{ODE}) for all odd functions $f$. 
It follows that $W$ is also invariant under the dynamics of all differences of systems of the form (\ref{ODE}),
which is all ODEs of the form ${\dot x}_i = f(x_i)$ with $f$ odd.

Let $\R^k \otimes \Delta_\mathcal{P}$ equal the intersection of all subspaces of the form $\R^k \otimes \Delta_\mathcal{Q}$ with $\Delta_\mathcal{Q}$ a polydiagonal subspace, that contain $W$.
We claim that $W$ contains an element $w$ satisfying $w_i = w_j$ if and only if $[i] = [j]$, and $w_i = -w_j$ if and only if $[i]^* = [j]$ in $\mathcal{P}$.
To see why, pick a pair $(i,j)$ for which $[i] \not= [j]$.
It follows that there is an element $w \in W$ for which $w_i \not= w_j$,
as this would otherwise contradict minimality of  $\R^k \otimes \Delta_\mathcal{P}$.
The same argument shows that for any pair $(i,j)$ for which $[i]^* \not= [j]$ an element 
$w \in W$ exists for which $w_i \not= -w_j$.
We conclude that the set of $w \in W$ satisfying $w_i \not= (-)w_j$ for all pairs $(i,j)$ with $[i]^{(*)} \not= [j]$ is given by $W$ with a finite number of strict subspaces cut out, and is therefore non-empty.
In what follows we fix such an element $w = (w_1, \dots, w_n) \in W$.

Next, we pick an element $u \in \R^k \otimes \Delta_\mathcal{P}$.
Let $g: \R^k \rightarrow \R^k$ be a function (odd or otherwise) satisfying $g(w_i) = u_i$ and $g(-w_i) = -u_i$ for all $i$.
This is well-defined as $w_i = w_j$ implies $[i] = [j]$ and so $u_i = u_j$.
Likewise $w_i = -w_j$ implies $[i]^* = [j]$ and so $u_i = -u_j$.
We now define the odd function $f$ by $f(x) := (g(x) - g(-x))/2$.
It follows that $f(w_i) = (g(w_i) - g(-w_i))/2 = (u_i + u_i)/2 = u_i$ for all $i$,
so that the map $(x_1, \dots, x_n) \mapsto (f(x_1), \dots, f(x_n))$ sends $w$ to $u$.
As this latter map sends elements of $W$ into $W$, and because $w \in W$, 
we conclude that $u \in W$.
The element $u \in \R^k \otimes \Delta_\mathcal{P}$ was chosen freely,
from which it follows that $\R^k \otimes \Delta_\mathcal{P} \subseteq W$ 
and so $W = \R^k \otimes \Delta_\mathcal{P}$.
To show that $\Delta_\mathcal{P}$ is $M$-invariant, 
we pick nonzero $x, y \in \R^k$ such that $Hx = y$.
Setting $f=0$, invariance of $W = \R^k \otimes \Delta_\mathcal{P}$ implies that
\begin{equation*}
\sum_{j=1}^n M_{i,j} s_jy  = H \sum_{j=1}^n M_{i,j} s_jx  = H \sum_{j=1}^n M_{k,j} s_jx  = \sum_{j=1}^n M_{k,j} s_jy 
\end{equation*}
for all $s = (s_j) \in \Delta_\mathcal{P}$ and $i,k$ such that $[i] = [k]$. Likewise, we have
\begin{equation*}
\sum_{j=1}^n M_{i,j} s_jy  = H \sum_{j=1}^n M_{i,j} s_jx  = -H \sum_{j=1}^n M_{k,j} s_jx  = -\sum_{j=1}^n M_{k,j} s_jy 
\end{equation*}
for all $s = (s_j) \in \Delta_\mathcal{P}$ and $i,k$ such that $[i]^* = [k]$.
As $y$ is nonzero, we conclude that $\Delta_\mathcal{P}$ is indeed $M$-invariant.

If $W$ is dynamically invariant for all $f$ with $f(0) = 0$, then it is invariant for all odd $f$.
It follows that $W = \R^k \otimes \Delta_\mathcal{P}$ for some polydiagonal subspace $\Delta_\mathcal{P}$.
The choice $f(x)  =  \| x \|^2 \, \1$ shows that $\Delta_\mathcal{P}$ is a synchrony subspace or a minimally tagged anti-synchrony subspace. 

When there are no conditions on $f$, the constant map $f(x) = \1$ excludes 
anti-synchrony subspaces.
\end{proof}

\begin{remk}
\label{balancedALinvariant}
The concept of a balanced partition of a coupled cell network is central to the seminal work of \cite{Pivato}.  The notion has been generalized in several ways by several papers.
Suppose $G$ is a weighted cell network with cells $C$, then \cite{AD2020} defines four types of tagged partitions in terms of their invariance under the adjacency matrix $A$ or Laplacian matrix $L$ as follows. 
Note that a partition of the set of cells is essentially a tagged partition for which the partial involution is the empty function.
\begin{itemize}
\item
A partition $\mathcal P$ of $C$ is \emph{balanced} for $G$ if $\Delta_{\mathcal P}$ is an $A$-invariant synchrony subspace.
\item
A partition $\mathcal P$ of $C$ is \emph{exo-balanced} for $G$ if $\Delta_{\mathcal P}$ is an $L$-invariant synchrony subspace.
\item
A tagged partition $\mathcal P$ of $C$ is \emph{even-odd-balanced} for $G$ if $\Delta_{\mathcal P}$ is an $A$-invariant anti-synchrony subspace.
\item
A tagged partition $\mathcal P$ of $C$ is \emph{linear-balanced} for $G$ if $\Delta_{\mathcal P}$ is an $L$-invariant anti-synchrony subspace.
\end{itemize}
We have decided to use terminology such as ``$L$-invariant anti-synchrony subspace'' instead, as we deem it more transparent for the goals of this paper.
\end{remk}

\begin{remk}
It is interesting that the class of odd-balanced subspaces defined in \cite{NSS6} is the only type of invariant subspaces for networks that does not have a characterization 
in terms of $M$-invariance for some matrix or set of matrices, as far as we know.  
Nevertheless, the odd-balanced subspaces are proved to be evenly tagged in \cite{NSS6}.
\end{remk}

\begin{ex}
Consider the digraph with 2 vertices and no arrows.  
All of the polydiagonal subspaces in Example~\ref{polyR2} are $A$-invariant, and $L$-invariant,
since both are the $2 \times 2$ zero matrix.
System~(\ref{ODE}) is thus two uncoupled identical ODEs $\dot x_i = f(x_i)$ 
for $i \in \{1,2\}$.
One direction of the statements in Proposition~\ref{dynamicallyInvariant} is obvious.
Subspaces 
$\R^k \otimes \Delta_{\mathcal{P}_4}$ and $\R^k \otimes \Delta_{\mathcal{P}_5}$ 
are dynamically invariant for all $f$.  
If $f(0) = 0$ and $j\in\{0,2,3,4,5\}$, then $\R^k \otimes \Delta_{\mathcal{P}_j}$ is dynamically invariant.
If $f$ is odd, then $\R^k \otimes \Delta_{\mathcal{P}_j}$ for all $j \in \{0, \ldots, 5\}$.
Proposition~\ref{dynamicallyInvariant} also says that no other subspaces are
dynamically invariant for all functions $f$ in each of the three classes.
\end{ex}

\begin{remk}
\label{results_with_h}
Proposition \ref{dynamicallyInvariant} shows that 
$k = 1$ is the only case we need to consider to understand the dynamically invariant subspaces of System (\ref{ODE}).
While the restriction of linear coupling seems very restrictive, the dynamically  invariant subspaces of a coupled cell network with more general coupling are often characterized as $A$ or $L$-invariant generalized polydiagonal subspaces.
For example, the balanced subspaces of \cite{Pivato} are $A$-invariant synchrony subspaces, and these are dynamically invariant for all the admissible vector fields.  See also \cite{AD2020, NSS6}.
\end{remk}

\begin{figure}
\begin{center}
\begin{tabular}{c c c}
\includegraphics[width = 2.0in]{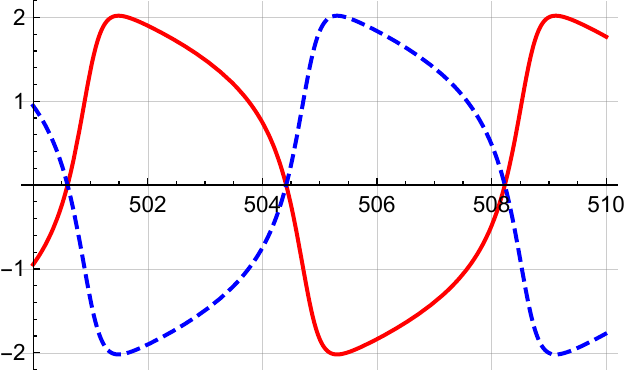} & 
\hspace{0.5cm} 
& \includegraphics[width = 2.0in]{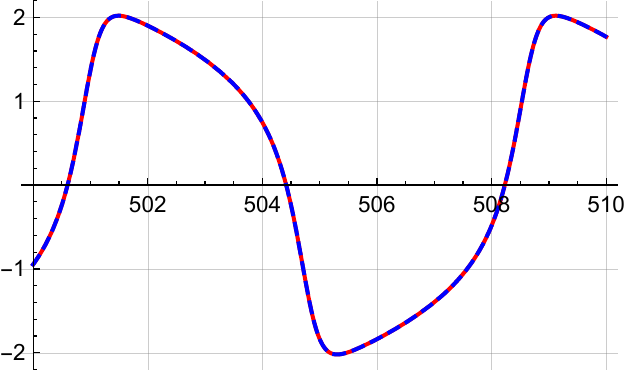}\\
$M = 0.5 A$ & & $M = 0.5 L$\\
 \\
\includegraphics[width = 2.0in]{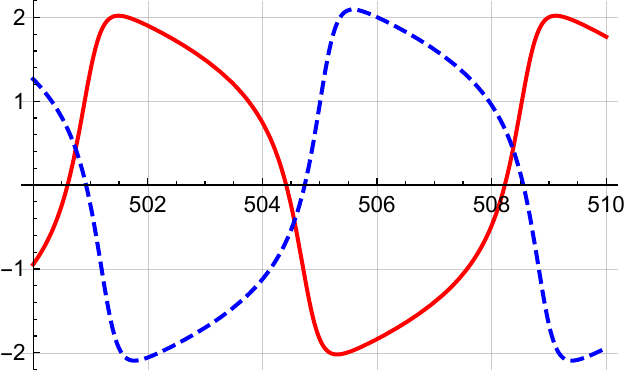} & & \includegraphics[width = 2.0in]{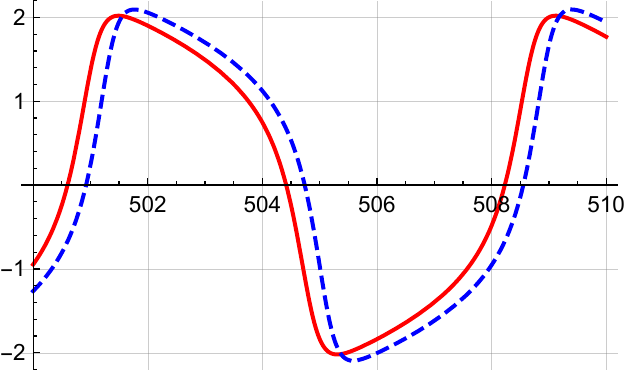}\\
$M = -0.5 A$ & & $M = -0.5 L$
\end{tabular}

\caption{\label{vdpFig}
Numerical solutions of System (\ref{coupled_vdp}) for 2 coupled van der Pol oscillators 
with $\varepsilon = 2$, and $M$ couplings
defined in Example~\ref{coupled_vdp_ex}.
The red curves are graphs of $u_1$ vs.\ $t$,  and the dashed blue curves are graphs of $u_2$ vs.\ $t$. 
}
\end{center}
\end{figure}

\begin{ex}
\label{coupled_vdp_ex}
Proposition~\ref{dynamicallyInvariant}(1) applies
to the coupled van der Pol oscillators in Example~\ref{van_der_Pol}.
Consider this system with
the weighted digraph shown below, along with its adjacency matrix $A$ and Laplacian matrix $L$.
The lattices of $A$- and $L$-invariant subspaces for the two cases are also shown.
\begin{center}
	\begin{tabular}{cccc}
		\begin{tabular}{c}
			\includegraphics{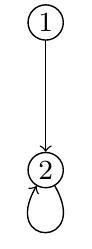}
		\end{tabular}  
& 
		\begin{tabular}{c}
			$A=\left[\begin{matrix}
			0 & 0\\
			1 & 1
			\end{matrix}\right]$
		\end{tabular}  
		\begin{tabular}{c}
			\includegraphics{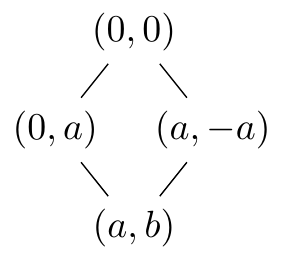}\\
		\end{tabular}

&
&
    	\begin{tabular}{c}
			$L=\left[\begin{matrix}
			0 & 0\\
			-1 & 1
			\end{matrix}\right]$
		\end{tabular}   
		\begin{tabular}{c}
			\includegraphics{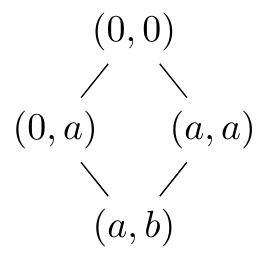}\\
		\end{tabular} \\
	\end{tabular}
	\par
\end{center}
Note that cell 1 is uncoupled. The question is whether cell 2 synchronizes with cell 1, anti-synchronizes, or neither.
Four numerical solutions to System~(\ref{coupled_vdp}) are shown in Figure~\ref{vdpFig}.
Mathematica notebooks for creating Figures~\ref{vdpFig} and \ref{LorenzFig}, 
as well as all of the lattices of invariant  subspaces shown in this paper,
are available at \cite{ourGitHub}.
For $M = 0.5 A$ the stable solution is in the anti-synchrony subspace,
and the synchrony subspace is not $A$-invariant.
Similarly, for $M = 0.5 L$ the stable solution is in the synchrony subspace,
and the anti-synchrony subspace is not $L$-invariant.
For both $M= -0.5 A$ and $M = -0.5 L$ the stable solution is neither synchronized nor anti-synchronized, and
the amplitude of cell 2 is slightly larger than the amplitude of cell 1.
The solutions all have the same random initial condition at $t = 0$, so the $u_1$ curves are the same.
However, the attractor in each case appears to be independent of initial conditions.
\end{ex}
Proposition~\ref{dynamicallyInvariant}(1) assumes that $f$ is odd.
The proposition applies to the coupled chaotic oscillators of \cite{Liu2006}.  
However,  it does not explain the anti-synchrony of the
coupled Lorenz attractors of Example~\ref{Lorenz}
observed in \cite{KimAntiSynch}.
Even though $f$ is not odd, it does satisfy
the condition
$f(Nx_i) = N f(x_i)$ 
for the matrix $N = \text{diag}(-1,-1,1)$.

For coupled cell networks in general,
the symmetry of the internal dynamics interacts with the symmetry of the network in a nontrivial way. 
It is natural to ask what invariant subspaces of System~(\ref{ODE}) are present for a given network for all functions $f$ with a prescribed symmetry.
This is a topic of great complexity \cite{Dionne_1996_I, Dionne_1996_II}, and we do not attempt to answer this question
in general.
However, we extend Proposition~\ref{dynamicallyInvariant}  with two results that apply to the coupled Lorenz equations.
The first result assumes a symmetry of the internal dynamics $f$, which need not be an involution.
\begin{prop}
\label{NH=H}
Let $F: (\R^k)^n \to (\R^k)^n$ be defined in terms of $f:\R^k \to \R^k$,  $H \in \R^{k \times k}$, and $M \in \R^{n \times n}$
as in System~(\ref{ODE}).
Assume that the matrix $N \in \R^{k \times k}$ satisfies 
$f(Nx_i) = Nf(x_i)$ for all $x_i \in \R^k$,
and $HN = NH = H$.
For each $\ell \in \{1, \ldots, n\}$, define $\gamma_\ell: (\R^k)^n \to (\R^k)^n$ by
$\gamma_\ell  (x) := (x_1, \ldots, N x_\ell, \ldots, x_n)$.
Then $F(\gamma_\ell(x)) = \gamma_\ell(F(x))$.
\end{prop}

\begin{proof}
Note that
$$
H \sum_{j = 1}^n M_{i,j} (\gamma_\ell (x))_j = H M_{i,\ell} N x_\ell + H \sum_{j \neq \ell} M_{i,j}  x_j = H \sum_{j = 1}^n M_{i,j} x_j
$$
since $HN = H$.
If $i \neq \ell$, then 
$$
F_i(\gamma_l(x)) = f(x_i) + H \sum_{j = 1}^n M_{i,j} x_j = F_i(x).
$$
Furthermore, since $H = NH$ it follows that 
\[
F_\ell(\gamma_\ell(x)) = f(N x_\ell) + H \sum_{j = 1}^n M_{\ell,j} x_j = Nf(x_\ell) + N H \sum_{j = 1}^n M_{\ell,j} x_j  = NF_\ell(x). \qedhere
\]
\end{proof}

\begin{remk}
If the hypotheses of Proposition~\ref{NH=H} hold for System~(\ref{ODE}), and a subspace $W \subseteq (\R^k)^n$ is dynamically invariant, 
then both $\gamma_\ell(W)$ and the fixed-point subspace $\{x \in W \mid \gamma_\ell(x) = x \}$ are dynamically invariant subspaces.
See \cite{GSS} for background about equivariant dynamical systems.
\end{remk}

\begin{ex}
Assume the hypotheses of Proposition~\ref{NH=H} hold,  with
$N^2 = I$ and $N \neq I$.
Then the involutions $\gamma_\ell$ commute, and for each subset $S$ of $\{1, \ldots, n\}$ we define the involution $\gamma_S$ in the 
obvious way.  These involutions form the image of a $\Z_2^n$ action on $(\R^k)^n$.  
The proposition implies that $F(\gamma_S(x)) = \gamma_S(F(x))$ for every $S \subseteq \{1, \ldots, n\}$.  
We say that $F$ has $\Z_2^n$ symmetry.  Of course, $F$ can have more symmetry.

Let $\DeltaP$ be an $M$-invariant synchrony subspace of $\R^n$ with $p$ classes in $\mathcal P$ of sizes 
$n_1, \ldots, n_p$, where
$n_1 + \cdots +  n_p = n$.
Then $W := \R^k \otimes \DeltaP$ is dynamically invariant for System~(\ref{ODE}) by Proposition~\ref{dynamicallyInvariant} 
and $\gamma_S(W)$ is also dynamically
invariant.  
There are $2^{n_1-1} \cdots 2^{n_p - 1} = 2^{n-p}$ distinct invariant subspaces
in the group orbit of $W$.
The exponent of $n_i - 1$ avoids double counting by not acting on the first cell
in each equivalence class of $\mathcal P$.
In addition to these invariant subspaces in the group orbit of $W$,
the intersection of any invariant subspace with the fixed point subspace $\{x \in (\R^k)^n \mid \gamma_S (x) = x\}$ is an invariant subspace for each subset $S$.
For example, the choice $S = \{1, \ldots, n\}$ yields the dynamically invariant subspace $\{x \in W \mid x_i = N x_i \mbox{ for all } i\}$.
\end{ex}
 
The following proposition does not require that $f$ is odd but in some cases describes a dynamically invariant subspace of $(\R^k)^n$ corresponding to
an $M$-invariant anti-synchrony subspace of $\R^n$.

\begin{prop}
\label{NH=-H}
Let $f:\R^k \to \R^k$,  $H \in \R^{k \times k}$, and $M \in \R^{n \times n}$ define System (\ref{ODE}), and
let $\DeltaP$ be an $M$-invariant polydiagonal subspace of $\R^n$.
Assume that the matrix $N \in \R^{k \times k}$ satisfies $N^2 = I$, 
$f(Nx_i) = Nf(x_i)$ for all $x_i \in \R^k$,
and $HN = NH = -H$.
Then the subspace
\[
\Delta_{\mathcal{P}}^N := \{ ( x_1,  \ldots , x_n ) \in (\R^{k})^n \mid   x_i = x_j \mbox { if } [i] = [j], \mbox{ and } x_i = N x_j \mbox{ if } [i] = [j]^* \}
\]
is dynamically invariant for System (\ref{ODE}). 
\end{prop}
\begin{proof}
Let $x = ( x_1,  \ldots , x_n ) \in \Delta_{\mathcal{P}}^N$ be given. 
We will first show that $(Hx_1,  \ldots , Hx_n) \in \R^k \otimes \DeltaP$. 
To this end, let $i,j \in \{1, \dots, n\}$ be such that $[i] = [j]$.
It follows that $x_i = x_j$ and so $Hx_i = Hx_j$.
If on the other hand we have $[i] = [j]^*$, then $x_i = N x_j$.
It follows that $Hx_i = HNx_j = -Hx_j$,
and we conclude that indeed $(Hx_1,  \ldots , Hx_n) \in \R^k \otimes \DeltaP$.
Because of this, we see that likewise 
\begin{align*}
\bigg(H\sum_{k=1}^n M_{1,k}x_k, \dots ,H\sum_{j=1}^n M_{n,k}x_k\bigg) = \bigg(\sum_{k=1}^n M_{1,k}Hx_k, \dots ,\sum_{j=1}^n M_{n,k}Hx_k\bigg) \in
\R^k \otimes \DeltaP\, .
 \end{align*}
Therefore, if we have $[i] = [j]$ then also $H\sum_{k=1}^n M_{i,k}x_k = H\sum_{k=1}^n M_{j,k}x_k$. 
If $[i] = [j]^*$ then $H\sum_{k=1}^n M_{i,k}x_k = -H\sum_{k=1}^n M_{j,k}x_k = NH\sum_{k=1}^n M_{j,k}x_k$.
We conclude that 
\begin{align}\label{1partdeltapninv}
\bigg(H\sum_{k=1}^n M_{1,k}x_k, \dots ,H\sum_{j=1}^n M_{n,k}x_k\bigg)  \in  \Delta_{\mathcal{P}}^N\, .
 \end{align}
 Of course $x_i = x_j$ implies $f(x_i) = f(x_j)$
 and $x_i = Nx_j$ gives $f(x_i) = f(Nx_j) = Nf(x_j)$.
 From this we see that likewise
 \begin{align}\label{2partdeltapninv}
(f(x_1), \dots, f(x_n)) \in  \Delta_{\mathcal{P}}^N\, .
\end{align}
The result 
follows from Equations (\ref{1partdeltapninv}) and (\ref{2partdeltapninv}).
\end{proof}

Note that $\Delta_{\mathcal P}^N = \R^k \otimes \DeltaP$ if $N = -I$ or $ \DeltaP$ is a synchrony subspace.

\begin{ex}
\label{Lorenz2Hs}
Two different coupling matrices illustrate Propositions~\ref{NH=H} and ~\ref{NH=-H} 
for the Lorenz equation $f$ defined in Example~\ref{Lorenz}.
Let
$$
H_+ = \left[ \begin{smallmatrix} 0 & 0 &  0 \\ 0 & 0 &  0 \\  0 & 0 &  1  \end{smallmatrix} \right],  \ 
H_- = \left[ \begin{smallmatrix}  0 & 0 &  0 \\ 0 & 1 &  0 \\  0 & 0 &  0 \end{smallmatrix} \right], \ \mbox{and }
N = \left[ \begin{smallmatrix} -1 & 0 &  0 \\ 0 & -1 &  0 \\  0 & 0 &  1 \end{smallmatrix} \right].
$$
Note that $f(N x_i) = N f(x_i)$ and $H_+N = N H_+  = H_+$, so the hypotheses of Proposition~\ref{NH=H} applies if
$H = H_+$.  Furthermore $N^2 = I$ and $H_- N = N H_-  = H_-$, so Proposition~\ref{NH=-H} applies if $H = H_-$.
System~(\ref{ODE}) in the two cases $H = H_+$ and $H = H_-$ respectively, is
\begin{eqnarray}
\label{HzLorenz} 
\dot{u}_i = 10(v_i - u_i) , \quad \dot{v}_i = u_i (28 - w_i) - v_i, \quad
\dot{w}_i = u_i v_i - {\textstyle \frac 8 3} w_i + \sum_{j = 1}^n M_{i,j} w_j , \\
\label{HyLorenz} 
\dot{u}_i = 10(v_i - u_i) , \quad \dot{v}_i = u_i (28 - w_i) - v_i + \sum_{j = 1}^n M_{i,j} v_j , \quad
\dot{w}_i = u_i v_i - {\textstyle \frac 8 3} w_i.
\end{eqnarray}
Note that we use $x_i = (u_i, v_i, w_i) \in \R^3$ in place of the traditional $(x_i,y_i,z_i) \in \R^3$ to respect our notation of $x_i \in \R^k$.

Figure~\ref{LorenzFig} shows solutions to these systems with two-way symmetric Laplacian coupling.
The digraph for this coupling is shown below, along with the Laplacian matrix $L$ and the lattice of $L$-invariant
polydiagonal subspaces:
\begin{center}
	\begin{tabular}{ccccc}
		\begin{tabular}{c}
			\includegraphics{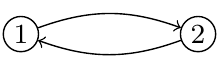}
		\end{tabular} & &
		\begin{tabular}{c}
			$L=\left[\begin{matrix}
			1 & -1 \\
			-1 & 1 \\
		\end{matrix}\right]$\\
		\end{tabular} & &
		\begin{tabular}{c}
			\includegraphics{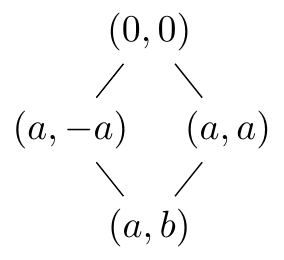}\\
		\end{tabular}\\
	\end{tabular}
\end{center}
Since $L$ has zero row sums, the diagonal subspace of $\R^2$ is $L$-invariant, 
and Proposition~\ref{dynamicallyInvariant} part (3) says that the subspace 
with typical element 
$\big ( (u_1,v_1,w_1), (u_1,v_1,w_1) \big )$ 
is dynamically invariant for both System~(\ref{HzLorenz}) and (\ref{HyLorenz}) with 
the network coupling matrix $M = \kappa L$.
The \emph{synchronous solution} in this subspace is two copies 
of the uncoupled Lorenz attractor
since the coupling terms vanish.
This synchronized solution is observed to be stable if $\kappa = -2$ and unstable if $\kappa = 2$ 
for both System~(\ref{HzLorenz}) and (\ref{HyLorenz}).
The Master Stability Function \cite{PhysRevE_Pecora09} is an efficient
way to compute the stability of the synchronous solution for any network with Laplacian coupling.

Both of the solutions shown in Figure~\ref{LorenzFig} are \emph{anti-synchronous solutions} in the dynamically invariant subspace with typical element 
$\big (  (u_1,v_1,w_1), (-u_1,-v_1,w_1) \big )$.
In the left figure, the anti-synchronous solution is related to the synchronous solution by the symmetry $(x_1, x_2) \mapsto (x_1, N x_2)$ 
of System~(\ref{HzLorenz}), as described in
Proposition~\ref{NH=H}.  
Both the synchronous solution (not shown) and the anti-synchronous solution (shown in the left of Figure~\ref{LorenzFig}) are observed to be stable
for System~(\ref{HzLorenz}) with $M = -2L$.
This is exactly the anti-synchrony reported in \cite{KimAntiSynch}.
 
The anti-synchronous solution shown at the right of Figure~\ref{LorenzFig} is in the same dynamically invariant
subspace, but for a different reason.  The subspace is dynamically invariant for System~(\ref{HyLorenz}),
as predicted by 
Proposition~\ref{NH=-H} when $\DeltaP$ is the $L$-invariant anti-synchrony subspace 
of $\R^2$ with typical element $(a, -a)$.   
This anti-synchronous solution is apparently stable, 
and coexists with the unstable synchronous solution of System~(\ref{HyLorenz}) when $M = 2L$. 
 This type of anti-synchronous solution has not been reported before, to our knowledge.
\begin{figure}
\begin{center}
\begin{tabular}{c c c}
\includegraphics[width = 2.7in]{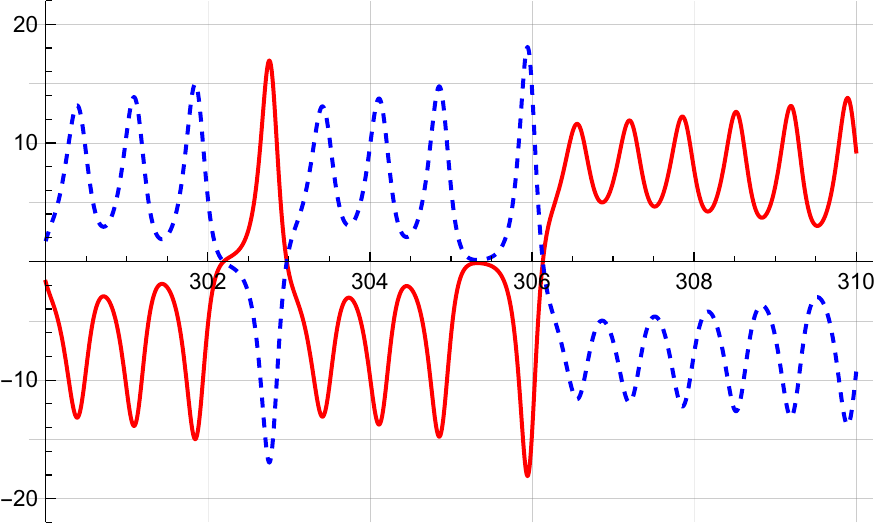} & &\includegraphics[width = 2.7in]{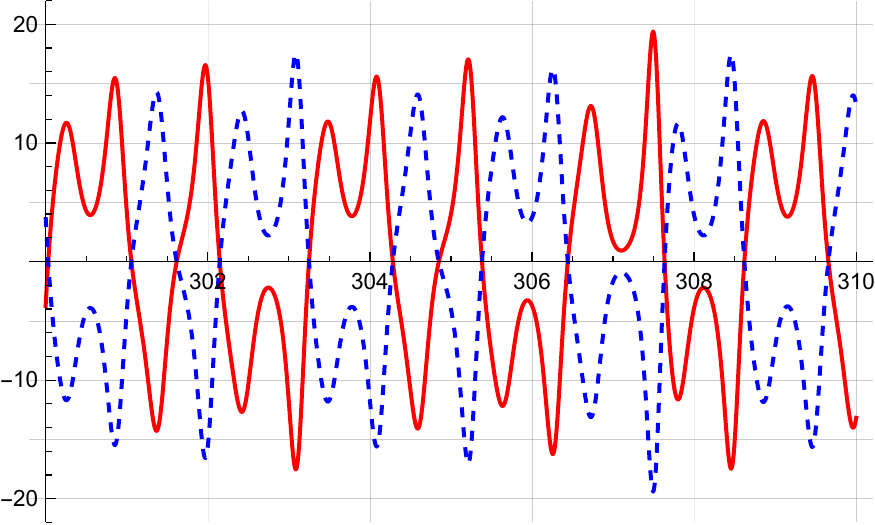} 
\end{tabular}
\caption{\label{LorenzFig}
Numerical solutions with random initial conditions for System~(\ref{HzLorenz}) on the left with $M = -2L$,
and System~(\ref{HyLorenz}) on the right with $M = 2L$, where $L$ is defined in Example~(\ref{Lorenz2Hs}).
The red curves are graphs of $u_1$ vs.\ $t$,  and the dashed blue curves are graphs of $u_2$ vs.\ $t$. 
}
\end{center}
\end{figure}
\end{ex}

\section{A Lemma about $M$-Invariant Subspaces}\label{ALemmaaboutMInvariantSubspaces}

Let $W$ be a subspace of $\R^n$ and $M \in \R^{n\times n}$. We say that $W$ is \emph{$M$-invariant} if $M W \subseteq W$.
Note that we are using the left action of $M$. 

To generalize the results and techniques from the previous section to non-symmetric matrices, we will look at both $M$ and $M^T$. 
Note that $M$ and $M^T$ are similar,  see \cite{Kaplansky}, hence they have the same eigenvalues, 
with the same geometric and algebraic multiplicities.

The key Lemma in this section is the following.

\begin{lem}
\label{MainLemma}
Assume $\lam$ is a real eigenvalue of $M \in \R^{n\times n}$ with geometric multiplicity 1. 
Let $v_L$ and $v_R$ be eigenvectors of $M^T$ and $M$, 
respectively, with eigenvalue $\lam$.
If $W$ is an $M$-invariant subspace, then $v_R \in W$ or $v_L \in W^\perp$.
\end{lem}

\begin{proof}
If $u \in \R^n$, then 
$$
v_L \cdot (M-\lam I)u = (M-\lam I)^Tv_L \cdot u =  \0 \cdot u = 0.
$$
As $W$ is an $M$-invariant subspace, we also see that $(M-\lam I)W \subseteq W$. The 
restriction $(M-\lam I)|_W: W \rightarrow W$ either has a non-trivial kernel, 
or is bijective by the rank-nullity theorem.
In the former case,
we have $v_R \in W$ since $\lam$ is simple.
In the latter case, $v_L \in W^\perp$ 
since  $v_L\in \im(M-\lambda I)^\perp$ by the displayed equation above.
\end{proof}

\begin{remk}
The subscript on $v_L$ indicates that $v_L^T$ is a left eigenvector of $M$.  That is,
$v_L^T M = \lam v_L^T$.
\end{remk}

In the case that $M = M^T$, Lemma~\ref{MainLemma} says that every simple eigenvector is either contained in a given $M$-invariant
subspace, or orthogonal to that subspace.  If the eigenvalues of a symmetric matrix are all simple, the invariant subspaces are particularly easy to compute, as shown in the next example.

\begin{ex}
\label{LapDirichlet}
Consider the weighted digraph  with its adjacency matrix $A$ and lattice of $A$-invariant polydiagonal subspaces shown below:
\begin{center}
\begin{tabular}{ccccc}
\begin{tabular}{c}
\includegraphics{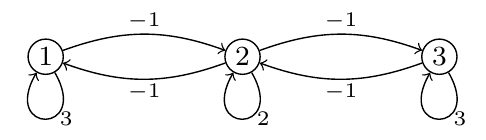}\\
\end{tabular} &  & %
\begin{tabular}{c}
$A=\left[\begin{matrix}3 & -1 & 0\\
-1 & 2 & -1\\
0 & -1 & 3
\end{matrix}\right]$\\
\end{tabular} &  & %
\begin{tabular}{c}
\includegraphics{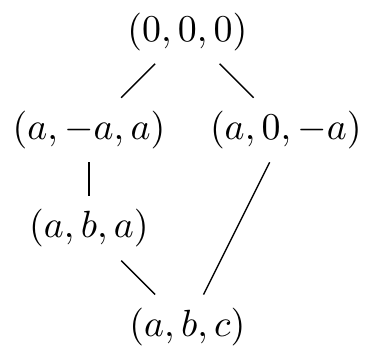}\\
\end{tabular}\\
\end{tabular}
\end{center}
The eigenvalues and eigenvectors of $A=A^T$ are
$$
\lam_1 = 1, v^{(1)} = (1,2,1) \qquad
\lam_2 = 3, v^{(2)} = (1,0,-1) \qquad
\lam_3 = 4, v^{(3)} = (1, -1, 1).
$$
Since all of the eigenvalues are simple, $W$ is an $A$-invariant subspace if and only if $W$ is spanned by a subset of the eigenvectors.
The following table shows that exactly 5 of the 8 $A$-invariant subspaces are polydiagonal subspaces,
and justifies the Hasse diagram of $A$-invariant polydiagonal subspaces shown above.
$$
\begin{array}{llll}
A\text{-invariant } W & \{i \mid v^{(i)} \in W\} & \{i \mid v^{(i)} \in W^\perp\} & \text{type of subspace}\\
\hline
(0,0,0) & \emptyset & \{1,2,3\} & \text{trivial}\\
(a, 2a,a) & \{1\} & \{2,3\} & \text{not polydiagonal}\\
(a,0,-a) & \{2\} & \{1, 3\} & \text{evenly tagged}\\
(a,-a,a) &\{3\} & \{1,2\} & \text{fully tagged}\\
(a,b,-a-2b) & \{2,3\} & \{1\} & \text{not polydiagonal}\\
(a, b, a) & \{1,3\} & \{2\} & \text{synchrony}\\
(a, b, -a+b) & \{1,2\} & \{3\} & \text{not polydiagonal}\\
(a,b,c) & \{1,2,3\} & \emptyset & \text{synchrony} \\
\hline
\end{array}
$$

For the symmetric matrix $A$, Lemma~\ref{MainLemma} states that $v^{(i)} \in W$ or $v^{(i)} \in W^\perp$ for each $i$, 
as this table illustrates.
Note that $(a,-a,a)$ is the typical element of an anti-synchrony subspace that is fully but not evenly tagged.

This example also illustrates that any matrix is the adjacency matrix of some weighted digraph.  
The complicated digraph was constructed to have the adjacency matrix $A$, 
which is of interest in the context of numerical differential equations.  
We can approximate a function $u: [0,1] \to \R$ by the values of $u$ at the midpoints of $n$ subintervals.  
For $n=3$, we approximate $u$ by the vector $\hat u = (u\left (\frac 1 6 \right), u\left ( \frac 1 2 \right ), u\left(\frac 5 6 \right )) $.  The map $u \mapsto -u''$  for functions with boundary conditions $u(0) = u(1) = 0$ is approximated by $\hat u \mapsto 9 A \hat u$.
The factor of $9$ is $1/h^2$, where $h = 1/3$ is the length of each subinterval \cite{NSS}.
\end{ex}

Note that $M$-invariant subspaces
that are not polydiagonal subspaces do not have any relevance for coupled cell networks.
In the remainder of the paper we will only consider the $M$-invariant polydiagonal subspaces.
Furthermore, it
is not always possible to list all of the $M$-invariant subspaces, as we did in Example~\ref{LapDirichlet}
because there are uncountably many if an eigenvalue of $M$ has geometric multiplicity greater than one.
Instead, we consider all 
of the polydiagonal subspaces in $\R^n$, and do a simple calculation to see which are $M$-invariant. 
This is feasible for $n \leq 8$.  The algorithm of \cite{Aguiar&Dias}, modified for weighted cell networks in \cite{AD2020} handles the multiple eigenvalues, and is feasible up to about $n=15$.

In the next example we consider a non-symmetric adjacency matrix.

\begin{ex}
\label{GandGT}
Consider the digraph of Network 3 in \cite[Figure 5]{leite} with its adjacency matrix $A$ and lattice of $A$-invariant polydiagonal subspaces shown below:
\begin{center}
	\begin{tabular}{ccccc}
		\begin{tabular}{c}
			\includegraphics{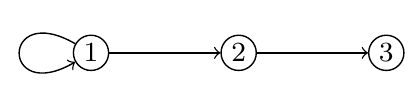}\\
		\end{tabular} &  & %
		\begin{tabular}{c}
			$A=\left[\begin{matrix}
			1 & 0 & 0\\
			1 & 0 & 0\\
			0 & 1 & 0
			\end{matrix}\right]$\\
		\end{tabular} &  & %
		\begin{tabular}{c}
			\includegraphics{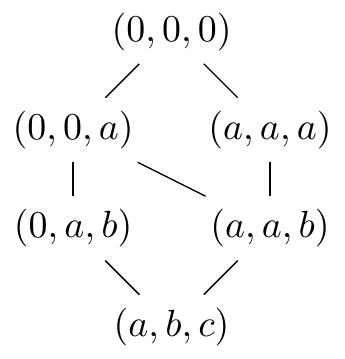}\\
		\end{tabular}\\
	\end{tabular}
\end{center}
The eigenvalues of $A$ are $\lam_0= 0$ with algebraic multiplicity 2 and geometric multiplicity 1, and $\lam_1 = 1$, with algebraic and geometric multiplicity 1.
The eigenvectors of $A$ and $A^T$ are
$$
\lambda_0=0, v_R^{(0)}=(0,0,1), v_L^{(0)} = (1,-1,0)  \qquad \lambda_1=1, v_R^{(1)} = (1,1,1), v_L^{(1)} = (1, 0, 0),
$$
respectively.  The compliance with Lemma~\ref{MainLemma} is shown in the table below:


\begin{center}
\begin{tabular}{llll}
$A$-invariant $\Delta$ & $\{i \mid v_R^{(i)} \in \Delta \}$ & 
 $\{i \mid v_L^{(i)} \in \Delta^\perp \}$ & \text{type of subspace}\\
\hline
$(0,0,0)$ & $\emptyset$ & $\{0, 1\} $ & \text{trivial} \\
$(0,0,a)$ &$\{0\} $ & $\{0, 1\} $ & \text{minimally tagged}\\
$(a,a,a)$ &$\{1\}$& $\{0\} $ & \text{synchrony}\\
$(0,a,b)$ & $\{0\} $ & $\{1\} $ & \text{minimally tagged}\\
$(a,a,b)$ & $\{0, 1\} $ & $\{0\} $ & \text{synchrony}\\
$(a,b,c)$ & $\{0, 1\} $ & $\emptyset$ & \text{synchrony} \\
\hline
\end{tabular}
\end{center}
Lemma~\ref{MainLemma} says that $v_R \in \Delta$ or $v_L \in \Delta^\perp$.  
This example shows that both might be true, since the second and fifth rows of the table tell us that 
$v_R^{(0)}\in \Delta$ and $v_L^{(0)} \in \Delta^\perp$.
\end{ex}

\begin{remk}
\label{F-P}
If $A$ is an irreducible matrix with non-negative entries, then the Frobenius-Perron Theorem for irreducible non-negative matrices tells us that $A$ has a simple eigenvalue $\la$ and a corresponding eigenvector with only positive entries. See for instance Chapter 2, in particular Theorems 1.4 and 2.7, of \cite{BermanPlemmons}.  Note that the transpose of an irreducible matrix is again irreducible, from which we see that positive left- and right-eigenvectors $v_L$ and $v_R$ exist for $\la$. Lemma \ref{MainLemma} then tells us that any $A$-invariant subspace $W$ satisfies $v_R \in W$ or $v_L \in W^\perp$. However, a synchrony subspace cannot be orthogonal to a positive vector, whereas an anti-synchrony subspace cannot contain a positive vector. Hence, we conclude that any synchrony subspace of $A$ has to contain $v_R$, while any anti-synchrony subspace has to be orthogonal to $v_L$. Note that this holds in the special case where $A$ is the adjacency matrix of an unweighted, strongly connected network.   
\end{remk}

\begin{ex}
Consider the digraph of Network~15 in \cite[Figure 5]{leite} with its adjacency matrix $A$ and lattice of $A$-invariant polydiagonal subspaces shown below:
\begin{center}
	\begin{tabular}{ccccc}
		\begin{tabular}{c}
			\includegraphics{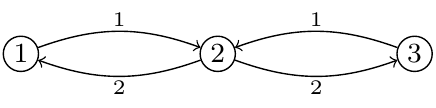}\\
		\end{tabular} &  & %
		\begin{tabular}{c}
			$A=\left[\begin{matrix}
			0 & 2 & 0\\
			1 & 0 & 1\\
			0 & 2 & 0
			\end{matrix}\right]$\\
		\end{tabular} &  & %
		\begin{tabular}{c}
			\includegraphics{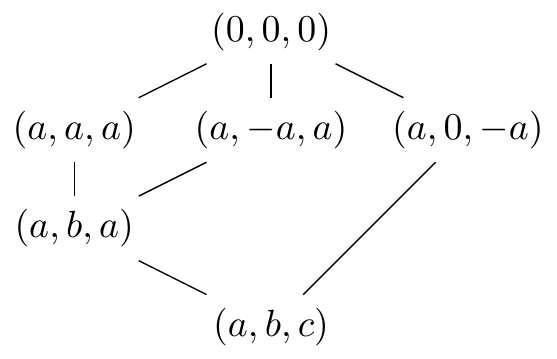}\\
		\end{tabular}\\
	\end{tabular}
\end{center}
The eigenvalues of $A$ are $\pm 2$ and 0.  
Remark~\ref{F-P} applies to this weighted, strongly connected digraph with non-negative weights.
The Frobenius-Perron eigenvalue is $\lam = 2$ and the associated eigenvectors are
$v_R = (1,1,1)$ and $v_L = (1,2,1)$. 
As predicted, each of the $A$-invariant synchrony subspaces contains $(1,1,1)$ 
and each of the $A$-invariant anti-synchrony subspaces is orthogonal to $(1,2,1)$.
\end{ex}

\section{Results for Matrices with Constant Column Sums}\label{Results_for_Matrices_with_Constant_Column_Sums}
 
The most general definition of a \emph{regular network} is a weighted digraph where the row sums of the weighted adjacency matrix $A$ are all the same, see \cite{AD2020}.  
This means that the sum of the weights of all the arrows coming into each vertex is the same for any vertex.
As a result, the synchrony subspace $\spn(\1)$ is $A$-invariant.

Our main applications of Lemma~\ref{MainLemma} concern what we might call \emph{output-regular networks}.  In this case the 
adjacency matrix has constant column sums.  This means that the sum of the weights of all the arrows with tails at each vertex is the same.  In an unweighted digraph, this just means that the number of arrows leaving each vertex is the same.

Note that if $M$ has constant column sums, then $\1$ is an eigenvector of $M^T$ and
the following is an immediate corollary of Lemma~\ref{MainLemma}.

\begin{cor}
\label{corOfMain}
Let $M\in \R^{n \times  n}$ have constant column sums $\lam$,
so $\1$ is an eigenvector of $M^T$ with eigenvalue $\lam$.
Assume the geometric multiplicity of $\lam$ is 1, 
and let $v$ be an eigenvector of $M$ with eigenvalue $\lam$.
Then every $M$-invariant subspace $W$ satisfies
$v \in W$ or $\1 \in W^\perp$. 
\end{cor}

\begin{ex}
\label{exOfCorOfMain}
Consider the weighted digraph with its adjacency matrix $A$ and lattice of $A$-invariant polydiagonal subspaces shown below:
\begin{center}
	\begin{tabular}{ccccc}
		\begin{tabular}{c}
			\includegraphics{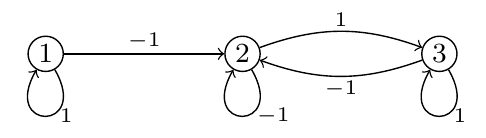}
		\end{tabular} &  & %
		\begin{tabular}{c}
			$A=\left[\begin{matrix}
			1 & 0 & 0\\
			-1 & -1 & -1\\
			0 & 1 & 1
			\end{matrix}\right]$\\
		\end{tabular} &  & %
		\begin{tabular}{c}
			\includegraphics{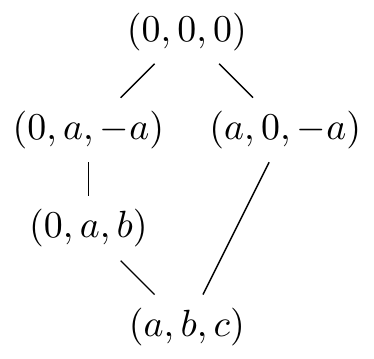}\\
		\end{tabular}\\
	\end{tabular}
	\par
\end{center}
The column sums of $A$ are all $0$, and the eigenvalue $\lam = 0$ has 
geometric multiplicity 1 with eigenvector $v = (0, 1, -1)$.   
As stated in Corollary~\ref{corOfMain}, every $A$-invariant polydiagonal subspace 
$\Delta$ has $v \in \Delta$ or $1 \in \Delta^\perp$.
Note that $v \perp \1$, so both are true for $\Delta = \spn(v) = \{(0, a, -a) \mid a \in \R\}$.
\end{ex}
Our main application of Corollary~\ref{corOfMain} is to polydiagonal subspaces.
We get a stronger conclusion if we assume $W$ is a polydiagonal subspace and make an additional assumption on the eigenvector $v$.

\begin{thr}
\label{main}
Let $M\in \R^{n \times  n}$ have constant column sums $\lam$, so that $\lam$ is an eigenvalue of $M$. Assume the geometric multiplicity of $\lam$ is 1, and a corresponding eigenvector $v$ satisfies $v_i +v_j \neq 0$ for all $i$ and $j$, including $i = j$.
Then an $M$-invariant polydiagonal subspace $\Delta_\calP$ is either a synchrony subspace that contains $v$ or an evenly tagged anti-synchrony subspace that does not contain $v$.
\end{thr}

\begin{proof}
Note that $\1$ is an eigenvector of $M^T$ with eigenvalue $\lambda$.

If $v \in \Delta_\mathcal{P}$, then the assumptions on $v$ imply that $[i]^*\ne [j]$ for all $i$ and $j$. Hence the partial involution of $\mathcal{P}$ is the empty function 
and $\Delta_\mathcal{P}$ is a synchrony subspace. 

If $v \not \in \Delta_\calP$, then Lemma~\ref{MainLemma} implies that $\1 \in \Delta_\mathcal{P}^\perp$.  In this case,
Proposition~\ref{orthodiagonal} implies that 
$\Delta_\mathcal{P}$ is an evenly tagged anti-synchrony subspace.
\end{proof}

\begin{remk}
\label{OddRemark}
One interesting consequence of Theorem \ref{main} pertains to (weighted) networks with an odd number of cells. To elaborate, if the matrix $M\in \R^{n \times  n}$ satisfies the conditions of Theorem \ref{main} and if $n$ is odd, then any polydiagonal subspace is either a synchrony subspace, or has at least one cell with vanishing internal state. This follows from the fact that an evenly tagged anti-synchrony subspace of an odd-dimensional subspace has to correspond to an involution with a fixed point. In particular, suppose we have a possibly non-linear dynamical system on a network with $n$ cells for which invariant subspaces are precisely the polydiagonal subspaces of the matrix $M$. This could for instance be the case if $M$ is the adjacency matrix or Laplacian of the network, see \cite{AD2020}. If we have reason to assume the state of our system is in a polydiagonal subspace, though not in a synchrony subspace (for instance by finding two cells with equal but opposing states), then we know for a fact that there has to be a cell with vanishing dynamics in the system. Note that we may not have any idea where this cell is situated in the network, nor does it have to be close to the two cells we have observed.
\end{remk}

\begin{remk}
\label{lowestDimSynch}
If $M$ and $v$ satisfy the hypotheses of Theorem~\ref{main},
then every $M$-invariant synchrony subspace
contains $v$.
The lowest dimensional $M$-invariant synchrony subspace 
is the polydiagonal subspace for
the coarsest $M$-invariant partition, as described in \cite{NSS7}.
At one extreme,
if $v = \1$, then $\spn ( \1 )$ is a one-dimensional $M$-invariant synchrony subspace. 
At the other extreme, if all $n$ components of $v$ are distinct, 
then the only $M$-invariant synchrony subspace is $\R^n$.
Examples with each of these extremes follow.
\end{remk}

\begin{ex}
\label{directedC3}
Consider the digraph with its adjacency matrix $A$ and lattice of $A$-invariant polydiagonal subspaces shown below:
\begin{center}
	\begin{tabular}{ccccc}
		\begin{tabular}{c}
			\includegraphics{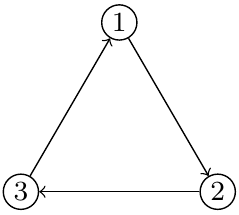}
		\end{tabular} &  & %
		\begin{tabular}{c}
			$A=\left[\begin{matrix}
			0 & 0 & 1\\
			1 & 0 & 0\\
			0 & 1 & 0
			\end{matrix}\right]$\\
		\end{tabular} &  & %
		\begin{tabular}{c}
			\includegraphics{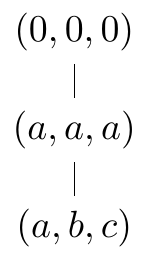}\\
		\end{tabular}\\
	\end{tabular}
	\par
\end{center}
The column sums of $A$ are all 1, and $v =\1$ is an eigenvector of $A$ with the simple eigenvalue 1.
Thus, Theorem~\ref{main} applies. The only evenly tagged anti-synchrony subspace is $\{\0\}$, and the other two $A$-invariant polydiagonal subspaces are synchrony subspaces that contain $v$. 
\end{ex}

\begin{ex}
\label{directedC4}
Consider the weighted digraph with its adjacency matrix $A$ and lattice of $A$-invariant polydiagonal subspaces shown below:
\begin{center}
	\begin{tabular}{ccccc}
		\begin{tabular}{c}
			\includegraphics{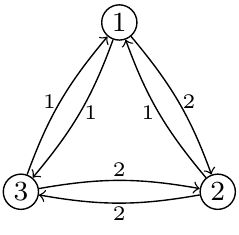}
		\end{tabular} &  & %
		\begin{tabular}{c}
			$A=\left[\begin{matrix}
			0 & 1 & 1\\
			2 & 0 & 2\\
			1 & 2 & 0
			\end{matrix}\right]$\\
		\end{tabular} &  & %
		\begin{tabular}{c}
			\includegraphics{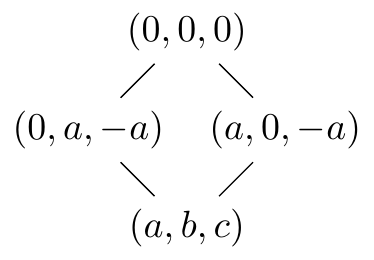}\\
		\end{tabular}\\
	\end{tabular}
	\par
\end{center}
The column sums of $A$ are all $3$, and an eigenvector of $A$ with the simple eigenvalue $\lam = 3$ is $v = (5,8,7)$.
Theorem~\ref{main} applies, and as
discussed in Remark~\ref{lowestDimSynch} the only $A$-invariant synchrony subspace is
$\R^3$ since all of the components of $v$ are distinct.
The other three $A$-invariant polydiagonal subspaces are evenly tagged anti-synchrony subspaces.
\end{ex}

Recall that a weighted digraph is called weight-balanced if for each vertex the indegree equals the outdegree. 

%
%

\begin{thr}
\label{input-output}
Let $L$ be the Laplacian matrix of a weight-balanced digraph.  
Assume that the eigenvalue $0$ of $L$ has geometric multiplicity one.
Then every  $L$-invariant anti-synchrony subspace is evenly tagged. 
\end{thr}

\begin{proof}
The Laplacian matrix $L$ is weight-balanced as the difference of the weight-balanced matrix $A$ and the weight-balanced diagonal degree matrix.
Since the row sums of any Laplacian matrix are 0, we conclude that the column sums of $L$ are all 0 too.
Hence, Theorem~\ref{main} applies with $\lambda = 0$ and $v = \1$.  
The 
result follows, since every $L$-invariant anti-synchrony subspace is evenly tagged.
\end{proof}

\begin{remk}\label{F_P2}
If $A$ is an irreducible matrix with non-negative off-diagonal entries,
then the corresponding Laplacian $L$ has a simple eigenvalue $0$. This result is known, but less so in our general setting.  To see that it still holds, note that for some large enough positive scalar $\mu \in \R$ the matrix  $\mu\Id-L$ has positive diagonal entries. Moreover, $\mu\Id-L$ agrees with $A$ on the off-diagonal entries, so that $\mu\Id-L$ equals the non-negative adjacency matrix of a strongly connected weighted digraph. As in Remark~\ref{F-P}, it follows from the Frobenius-Perron theorem that $\mu\Id-L$ has a simple eigenvalue $\la$ with a positive eigenvector $v$. Moreover, any non-negative eigenvector is necessarily a multiple of $v$, see \cite[Theorem 1.4]{BermanPlemmons}.  As $\mu\Id-L$ has an eigenvalue $\mu$ with eigenvector $\1$, we conclude that $\1$ is a multiple of $v$ and that $\mu = \la$. In particular, the eigenvalue $\mu$ of  $\mu\Id-L$ is simple, so that the eigenvalue $0$ of $L$ is as well.
\end{remk}

The following proposition, combined with Remark \ref{F_P2} above, can be very useful for showing that the Laplacian of a weight-balanced digraph satisfies the conditions of Theorem \ref{input-output}. The proof is only a small adaptation of the one commonly used for Eulerian digraphs, so we will not include it here. See for example \cite[Theorem 23.1]{WilsonBook} 
or \cite[Theorem 1.4.24]{WestBook}.


\begin{prop}
If a weight-balanced digraph is weakly connected with positive weights for all arrows that are not self-loops, then it is strongly connected.
\end{prop}

\begin{ex}
Consider the weighted digraph with its Laplacian matrix $L$ and lattice of $L$-invariant polydiagonal subspaces shown below:
\begin{center}
	\begin{tabular}{ccccc}
		\begin{tabular}{c}
			\includegraphics{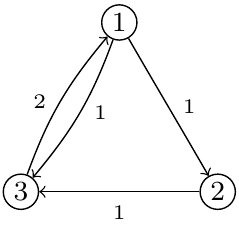}
		\end{tabular} &  & %
		\begin{tabular}{c}
			$L=\left[\begin{matrix}
			2 & 0 & -2\\
			-1 & 1 & 0\\
			-1 & -1 & 2
			\end{matrix}\right]$\\
		\end{tabular} &  & %
		\begin{tabular}{c}
			\includegraphics{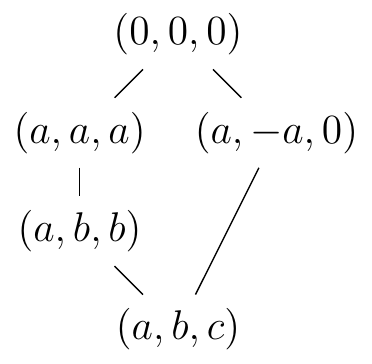}\\
		\end{tabular}\\
	\end{tabular}
	\par
\end{center}
The adjacency matrix of the weighted digraph is
$$
A = \left[\begin{smallmatrix}
0 & 0 & 2\\
1 & 0 & 0\\
1 & 1 & 0
\end{smallmatrix}\right].
$$
The second row sum and second column sum of $A$ are both 1.  All other row sums and column sums are 2.
Thus the weighted digraph is weight-balanced. Note that the $L$-invariant anti-synchrony
subspaces are all evenly tagged, as described by Theorem~\ref{input-output}. The $A$-invariant
polydiagonal subspaces are only $\{(0,0,0)\}$ and $\R^3$.  These are invariant for
every $3 \times 3$ matrix. 
\end{ex}

\begin{ex}Consider the weight-balanced digraph of Example~\ref{LapDirichlet}.  The $A$-invariant subspace with typical element $(a, -a, a)$ is not
evenly tagged, but this is not an $L$-invariant subspace.  The lattice of $L$-invariant subspaces is similar to the lattice of $A$-invariant subspaces, but
with $(a, -a,a) $ replaced by $(a,a,a)$. 
The $L$-invariant anti-synchrony subspaces $\{(0,0,0)\}$ and $\spn\{(1, 0, -1)\}$ are both evenly tagged.
\end{ex}

A weighted digraph is \emph{vertex transitive} if each vertex can be mapped to any other vertex by an automorphism of the weighted digraph.

The following result provides more examples of weight-balanced digraphs. 

\begin{lem}
\label{symimplieseul}
A vertex transitive  weighted digraph is weight-balanced.
\end{lem}

\begin{proof}
The transitivity ensures that the imbalance of every vertex is the same. The sum 
$$
\sum_{i \in V} b(i) = 0
$$ 
of the imbalances of the vertices is zero since every weight is counted twice with opposite sign. Thus the imbalance of every vertex must be 0.
\end{proof}

\begin{remk}
\label{vertexTransitive}
If the  indegree is the same for each vertex of a weighted digraph, then the adjacency matrix and the Laplacian add up to a multiple of the identity. 
Hence, both matrices have the same invariant subspaces. This is, for instance, the case for vertex transitive weighted digraphs. 
\end{remk}

A vertex transitive \emph{weighted Cayley digraph} can be constructed from a Cayley color digraph of a group by replacing the arrow-labels with weights. 
The automorphism group of the resulting weighted digraph is isomorphic to the original group if 
the arrow relabeling function is injective. 
Otherwise the relabeling can introduce additional automorphisms. 
We give two examples.

\begin{ex}
\label{Z7Cayley}
Consider the group $\langle \sigma \mid \sigma^7 = 1\rangle \cong \Z_7$.  
The weighted Cayley digraph using the generating set $\{\sigma, \sigma^{-1} \}$ and the 
arrow relabeling 
$\sigma \mapsto s$, $\sigma^{-1} \mapsto t$ is shown on the left.
The lattice of $A$-invariant polydiagonal subspaces when the weights satisfy $s \neq t$ is shown in the middle.
In this case, the automorphism group of the weighted Cayley digraph is isomorphic to
the original group $\Z_7$. 

If $s = t$, the automorphism group of the weighted Cayley digraph is isomorphic to $D_7$ due to the additional reflection symmetries.
The lattice for this case is shown on the right.
The boxes indicate group orbits of polydiagonal subspaces.  
One representative is shown, and the size of the group orbit (in this case, 7) is shown. 

The result of Theorem~\ref{input-output} holds for equal or unequal weights.  Since the number of vertices is odd and every
anti-synchrony subspace is evenly tagged,
the presence of any anti-synchronous pair implies at least one cell is in the 0 state,
as seen in the subspace $(a, b, c, 0, -c, -b, -a)$.
See Remark~\ref{OddRemark}.
\begin{center}
	\begin{tabular}{ccc}
		\begin{tabular}{c}
			\includegraphics{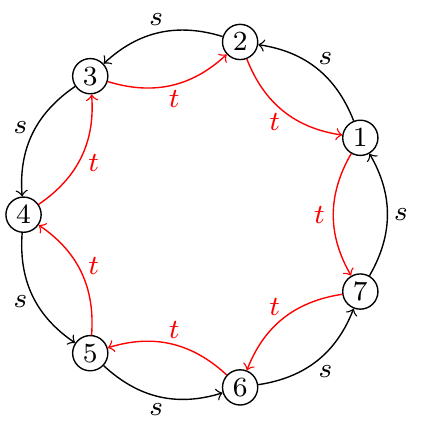}
		\end{tabular} &  
		\begin{tabular}{c}
		    \includegraphics{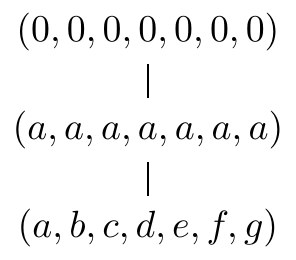} 
		\end{tabular} & 
		\begin{tabular}{c}
			\includegraphics{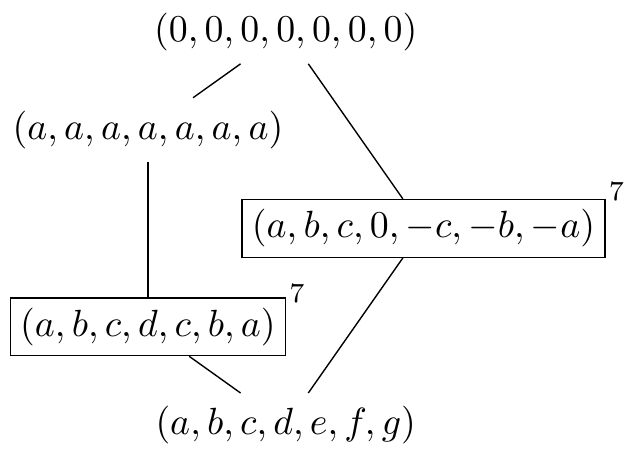}
		\end{tabular}
		\\ \\
		& $s \neq t$ & $s = t$ 
	\end{tabular}
\end{center}

\end{ex}

\begin{ex}
\label{D3Cayley}
Consider the $2k$-element group 
$ \langle \sigma, \tau \mid \sigma^2 =\tau^2 = (\sigma \tau)^k = 1\rangle \cong D_k$.
This is the symmetry group of a regular $k$-gon if $k > 2$.
The weighted Cayley digraph for $D_3$ using the generating set $\{\sigma, \tau \}$ and the 
arrow relabeling 
$\sigma \mapsto s$, $\tau \mapsto t$ is shown on the left.
For $s \neq t$ the automorphism group of this weighted Cayley digraph is $D_3$, and
the lattice of $A$-invariant polydiagonal subspaces is shown on the right. 
If $s = t$ the automorphism group of the weighted Cayley digraph is $D_6$, but 
we do not show the lattice of $A$-invariant polydiagonal subspaces in this case because it is so large.
A straightforward but lengthy calculation shows that there
are $31$ $A$-invariant polydiagonal subspaces in $15$ group orbits.
Among these, there are $18$ $A$-invariant anti-synchrony subspaces 
in $8$ group orbits, which are all evenly tagged following Theorem~\ref{input-output}.
\begin{center}
	\begin{tabular}{ccc}
		\begin{tabular}{c}
			\includegraphics{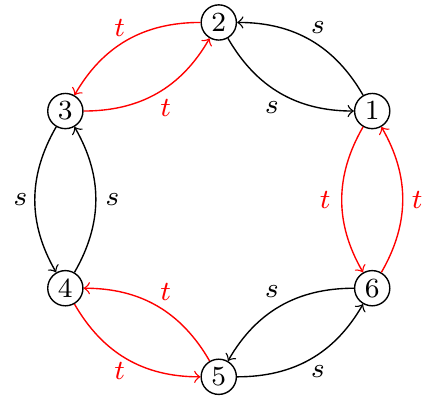}
		\end{tabular} & & 
		\begin{tabular}{c}
			\includegraphics{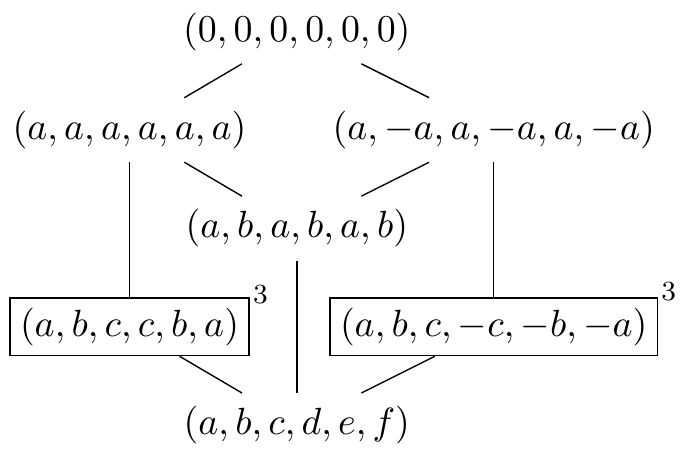}
		\end{tabular}
		\\ \\
		& & $s \neq t$
	\end{tabular}
\end{center}
\end{ex}

It was shown in 
\cite[Proposition 5.1]{NSS6} 
that every $L$-invariant polydiagonal subspace for the Laplacian matrix $L$ of a graph is either a synchrony subspace or
a fully tagged anti-synchrony subspace.  
Furthermore, it was
conjectured that every $L$-invariant anti-synchrony subspace is not just fully tagged, but is evenly tagged.

\begin{thr}[{\cite[Conjecture 5.3]{NSS6}}]
\label{conjecture53}
Let $L$ be the Laplacian matrix of a connected graph. Then every $L$-invariant anti-synchrony subspace is evenly tagged. 
\end{thr}

\begin{proof}
The digraph corresponding to a graph is trivially weight-balanced. 
It is well known,
and follows from Remark \ref{F_P2},
that the eigenvalue 0 of the Laplacian matrix of a connected graph is simple.
The result now follows from Theorem~\ref{input-output}.
\end{proof}

\begin{remk}
Recall that a graph is undirected (also called bidirectional) and unweighted by our definition.
Consider System (\ref{ODEwithL}),
where $f$ is odd and $L$ is the Laplacian matrix of a connected graph.
Propsition \ref{dynamicallyInvariant} and Theorem \ref{conjecture53} state that
a subspace $W \subseteq (\R^k)^n$ is dynamically invariant
for any such coupled cell network if and only if 
$W$ is an $L$-invariant synchrony subspace or 
$W$ is an $L$-invariant evenly tagged subspace.
\end{remk}

\begin{ex}
\label{Kn}
Let $L$ be the Laplacian matrix of the complete graph with $n$ 
vertices.
The eigenvalues of $L$ are $\lam = 0$, with algebraic multiplicity 1 and eigenvector $\1$,
and $\lam = n$ with geometric and algebraic multiplicity $n-1$, 
and eigenspace $\spn (\1 )^\perp$.
It is easy to show that a polydiagonal subspace $\Delta \subseteq \R^n$ is $L$-invariant if and only if 
$\Delta$ is a synchrony subspace or $\Delta$ is an evenly tagged anti-synchrony subspace. 
\end{ex}

The following ``non-example'' demonstrates the necessity of the hypothesis 
that the graph is connected.

\begin{ex}
\label{3v1e}
Consider the graph, represented as a digraph, with its Laplacian matrix $L$ and 
lattice of $L$-invariant polydiagonal subspaces shown below:
\begin{center}
	\begin{tabular}{ccccc}
		\begin{tabular}{c}
			\includegraphics{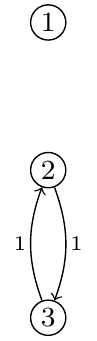}
		\end{tabular} & &
		\begin{tabular}{c}
			$L=\left[\begin{matrix}
			0 & 0 & 0\\
			0 & 1 & -1\\
			0 & -1 & 1
			\end{matrix}\right]$\\
		\end{tabular} & &
		\begin{tabular}{c}
			\includegraphics{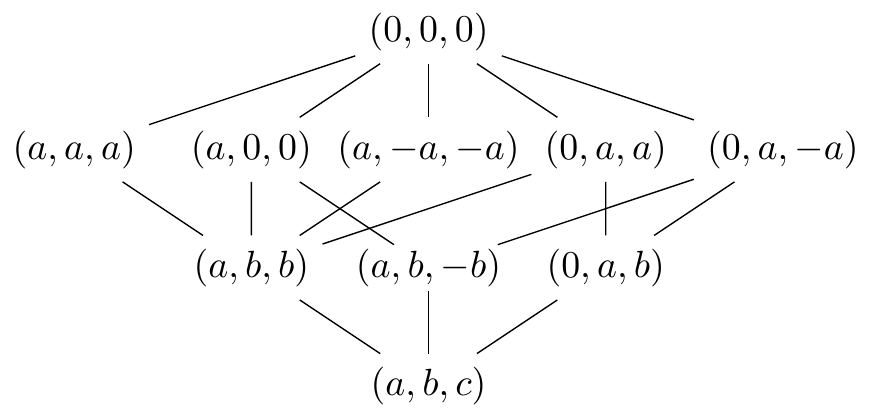}\\
		\end{tabular}\\
	\end{tabular}
\end{center}
Theorem~\ref{conjecture53} does not apply to this disconnected graph.  
Note that the eigenvalue 0 of $L$ has multiplicity 2 since the graph has 2 connected components.

Five of the ten $L$-invariant polydiagonal subspaces are anti-synchrony subspaces that are not evenly tagged.
These could not be present for a connected graph.

Note that every region in Figure~\ref{classification} is represented in this example.
The three synchrony subspaces have typical elements $(a,a,a)$,  $(a,b,b)$ and $(a,b,c)$.  
The four minimally tagged anti-synchrony subspaces have typical elements $(0,0,0)$, $(a,0,0)$, $(0, a, a)$ and $(0, a, b)$.
The two evenly tagged anti-synchrony subspaces have typical elements $(0,0,0)$ and $(0, a, -a)$.
The anti-synchrony subspace with typical element $(a, -a, -a)$ is fully but not evenly tagged.
The anti-synchrony subspace with typical element $(a, b, -b)$ is neither minimally nor fully tagged.
\end{ex}

\section{Counting Polydiagonal Subspaces}\label{Counting_polydiagonal_subspaces}
\label{sequences}

In this section, we count the number of polydiagonal subspaces in $\R^n$ of a given type.
It turns out many of the corresponding sequences are known in a different context, showing some surprising links to other areas of mathematics. 


We define the following sequences:
\begin{itemize}
\item
$p_n$, the number of polydiagonal subspaces of $\R^n$, is a Dowling number 
\cite[\href{https://oeis.org/A007405}{A007405}]{oeis}, as shown in Proposition~\ref{BtypeResult} or Proposition~\ref{egfP}.
\item
$s_n$, the number of synchrony subspaces of $\R^n$, is the Bell number $B_n$ \cite[\href{https://oeis.org/A000110}{A000110}]{oeis},
which counts the partitions of a set with $n$ elements.
\item
$a_n = p_n - s_n$ is the number of anti-synchrony subspaces of $\R^n$.
\item
$m_n$, the number of minimally tagged anti-synchrony subspaces of $\R^n$, 
is the 2-Bell number $B_{n-1,2}$ \cite[\href{https://oeis.org/A005493}{A005493}]{oeis}, 
which counts the number of partitions of the set with $n$ elements and one distinguished block.
In our case, the distinguished block is $P_0$. 
\item
$f_n$, the number of fully tagged polydiagonal subspaces of $\R^n$, is element $n$ of the sequence
\cite[\href{https://oeis.org/A307389}{A307389}]{oeis}, as shown in Proposition \ref{egfF}.
\item
$e_n$, the number of evenly tagged polydiagonal subspaces of $\R^n$ \cite[\href{https://oeis.org/A355194}{A355194}]{oeis}, is computed using Proposition~\ref{egfE} or \ref{recurrence}. 
\end{itemize}
The number of each type of polydiagonal subspace is shown in Figure~\ref{numberOfPolydiagonals}. 
These are computed by the code at \cite{ourGitHub}, using the results of Subsections~\ref{EGF} and \ref{recursion}.


\label{prop:}  

\begin{remk}
\label{mnisbn}
It is easy to see that $m_n = B_{n+1}-B_n$ since minimally tagged partitions of $\{1, \ldots, n\}$ are in one-to-one
correspondence with partitions of $\{1, \ldots, n+1\}$ for which $\{1, \ldots, n\} \cap [n+1]$ is not empty. 
Note that there are $B_n$ partitions of $\{1, \ldots, n+1\}$ for which $\{n+1\}$ is a class, as it remains to partition $\{1, \ldots, n\}$.  
\end{remk}

\begin{figure}
$$
\begin{array}{ll |rrrrrrrrr}
&  n &  0& 1 & 2 & 3 & 4 & 5 & 6 & 7 & 8 \\
 \hline
 \text{polydiagonal} & p_n & 1& 2 & 6 & 24 & 116 & 648 & 4\,088 & 28\,640 & 219\,920 \\
 \text{synchrony}& s_n  & 1& 1 & 2 & 5 & 15 & 52 & 203 & 877 & 4\,140 \\
 \text{anti-synchrony}& a_n & 0 & 1 & 4 & 19 & 101 & 596 & 3\,885 & 27\,763 & 215\,780 \\
 \text{minimally tagged}& m_n & 0 & 1 & 3 & 10 & 37 & 151 & 674 & 3\,263 & 17\,007 \\
 \text{fully tagged}& f_n & 1 &1 & 2 & 7 & 29 & 136 & 737 & 4\,537 & 30\,914 \\
 \text{evenly tagged}& e_n &1 & 1 & 2 & 4 & 13 & 41 & 176 & 722 & 3\,774 \\
\end{array}
$$
\caption{
\label{numberOfPolydiagonals}
The number of various types of polydiagonal subspaces of $\R^n$.
}
\end{figure}

\begin{remk}
Figure~\ref{numberOfPolydiagonals} shows how small $e_n$ is compared to $a_n$. 
This greatly strengthens the result of Theorem \ref{main} 
and its corollaries:
instead of checking all possible anti-synchrony subspaces for invariance,
only the smaller subset of the evenly tagged ones has to be considered.
\end{remk}

\subsection{Polydiagonal subspaces}

To show that the number of polydiagonal subspaces are the Dowling numbers, 
we start by giving an alternative way of characterizing polydiagonal subspaces. 
As in~\cite{Adler}, we define the following.


\begin{defi}
A  \emph{ $B$-type partition} of $\{-n, \dots, n\}$ is a partition $\mathcal{Q}$ of this set satisfying
\begin{enumerate}
\item If $Q = \{j_1, \dots, j_k\}$ is a class in $\mathcal{Q}$, then so is $-Q := \{-j_1, \dots, -j_k\}$.
\item There is precisely one class $Q_0 \in \mathcal{Q}$ for which $Q_0 = -Q_0$.
\end{enumerate}
\end{defi}

Note that $Q_0$ is necessarily the class containing the element $0$. 

The reason for introducing $B$-type partitions is as follows. Given a tagged partition $\mathcal{P}$ of $\{1, \dots, n\}$, we may create a $B$-type partition $\Phi(\mathcal{P})$ of $\{-n, \dots, n\}$ with classes given by 
\begin{itemize}
\item $P$ and $-P$ for each $P \in \mathcal{P}$ on which $*$ is not defined;
\item $P\cup -P^*$ for each $P \in \mathcal{P}$ for which $P^* \not= P$;
\item $P \cup \{0\} \cup -P$ if $P$ exists such that $P^* = P$, and $\{0\}$ otherwise.
\end{itemize}
One can verify that $\Phi(\mathcal{P})$ is indeed a $B$-type partition of $\{-n, \dots, n\}$. 

The number of $B$-type partitions of $\{-n, \dots, n\}$ is given by the so-called Dowling numbers, see for instance~\cite{Adler}.

\begin{ex}
	We return to the tagged partition of Example \ref{ex:polydiagonal}, given by
	$$
	\mathcal{P}=\{\{1\},\{2,4\},\{3\},\{5,6\}\}
	$$ 
	with $\{1\}^*=\{2,4\}$ and $\{5,6\}^*=\{5,6\}$. The corresponding $B$-type partition of $\{-6,  \dots,   6\}$ is
	$$
	\Phi(\mathcal{P})=\{\{-4, -2, 1\},\{-1,2,4\},\{3\},\{-3\},\{-6, -5, 0, 5,6\}\}\, .
	$$ 
\end{ex}

Given a $B$-type partition $\mathcal{Q}$ of $\{-n, \dots, n\}$, we construct a tagged partition $\Psi(\mathcal{Q})$ as follows. 
Given $Q \in \mathcal{Q}$ we define
\begin{equation*}
Q^+ := \{k \in Q\mid k>0\}
\end{equation*}
to be the set of positive elements of $Q$.
We then set
\begin{equation*}
\Psi(\mathcal{Q}) = \{Q^+\mid Q \in \mathcal{Q} \text{ and } Q^+ \not= \emptyset\}\, .
\end{equation*}
Note that this clearly gives a partition of $\{1, \dots, n\}$. The involution $*$ is defined on those classes $Q^+ \in \Psi(\mathcal{Q})$ for which $(-Q)^+ \not= \emptyset$, in which case we set $(Q^+)^* = (-Q)^+$. It not hard to see that this indeed gives a partial involution. 

One can verify that $\Phi$ and $\Psi$ are inverses, so we have the following result.

\begin{prop}
\label{BtypeResult}
The tagged partitions of $\{1, \dots, n\}$ are in one-to-one correspondence with the B-type partitions of $\{-n, \dots, n\}$.  Thus the number of polydiagonal subspaces of $\R^n$ are the Dowling numbers.
\end{prop}

\subsection{Freely tagged polydiagonal subspaces}

The following definition will be useful for counting various types of polydiagonal subspaces in Subsection~\ref{EGF}. Freely tagged subspaces are also interesting in their own right.

\begin{defi}
\label{freely}
Let $\mathcal{P}$ be a tagged partition of $C = \{1,\ldots,n\}$. The polydiagonal subspace $\Delta_{\mathcal{P}}$ is called \emph{freely tagged} if the partial involution has no fixed point.
\end{defi}

The typical element of a freely tagged subspace has no zero component. 
For example, $(a, b, a, -a)$ is the typical element of a freely tagged subspace, but $(a, b, 0, -a)$ is not.

\begin{ex}
All of the nontrivial $A$-invariant subspaces shown in Example~\ref{D3Cayley} are freely tagged.
\end{ex}

Another class of coupled cell networks where freely tagged subspaces are important is when the domain of $f$ does not contain $0 \in \R^k$. 
Given a subset $D \subseteq \R^k$ and a tagged partition $\mathcal{P}$, we define the subset
\[
D_{\mathcal{P}} :=\{x\in D^n \mid x_{i}=x_{j}\text{ if }[i]=[j]\text{ and }x_{i}=-x_{j}\text{ if }[i]^{*}=[j]\} \subseteq \R^k \otimes \DeltaP \, .
\]
Note that $D_{\mathcal{P}}$ is not a vector space if $D$ is not a vector space, 
and it might be empty.
For example, if $\calP$ has a non-trivial involution, and $-D\cap D= \emptyset$, then $D_\calP$ is empty.
Likewise, if the involution has a non-empty fixed point $P_0$, and $0 \not \in D$,  then the set $D_{\mathcal{P}}$ is empty.

We can give a version of the easy direction of Proposition~\ref{dynamicallyInvariant}.
Note that the domain of an odd function $f:D\to V$ satisfies  $D = -D$.

\begin{prop}
Let $f: D \to \R^k$ with $\emptyset\ne D\subseteq\R^k$, and let $\DeltaP$ be an $M$-invariant polydiagonal subspace of $\R^n$.
Then $D_{\mathcal{P}}$ is a non-empty dynamically invariant subset of System~$(\ref{ODE})$ if one of the following conditions holds:
\begin{enumerate}
\item 
$\DeltaP$ is freely tagged and $f$ is odd with $0 \not \in D $;
\item 
$\DeltaP$ is a synchrony subspace. 
\end{enumerate}
\end{prop}

\begin{proof}
The result follows from Equation~(\ref{Minv}), using the local existence of solutions.
The assumptions on $D$ guarantee that $D_{\mathcal P}$ is non-empty.
\end{proof}

\begin{ex}
Consider $n$ coupled cells with the dynamics 
$$
{\ddot u}_i = -{\dot u}_i -u_i + \frac 1 {u_i}  + (A u)_i .
$$
This can be put into the form of System~(\ref{ODE}) with the same $H$ matrix defined in Example~\ref{van_der_Pol} but with modified
internal dynamics.
The domain of $f: D \to \R^2$ defined by $f(u,v) = (v, -v  -u + 1/u )$ is $D = \{(u,v) \in \R^2 \mid u \neq 0\}$.
While $f$ is odd, $0 \not\in D$. If $\DeltaP$ is a freely tagged $A$-invariant subspace, then $D_{\mathcal P}$ is a dynamically invariant set 
that is not a subspace of $(\R^2)^n$.
\end{ex}

\subsection{Exponential generating functions}
\label{EGF}
Polydiagonal subspaces correspond to partitions. So the natural way
to count them is to compute their exponential generating functions.
Recall that the \emph{Exponential Generating Function} (EGF)
of a sequence $(a_{n})$ is the formal power series 
\[
A(x):=\sum_{n=0}^{\infty}a_{n}\frac{x^{n}}{n!}.
\]
The sequence $(a_n)$ can be easily found from $A(x)$ using Taylor expansion. 
Most computer algebra systems have a built in command for this.  
For example, the following code computes the first few terms of the sequence $(1, 1, \ldots )$ whose EGF is $e^x$:
\vspace{-2.5mm}
\begin{lstlisting}[basicstyle={\footnotesize\ttfamily},frame=single,backgroundcolor=\color{lightgray!20},title={Sage:}]
p=taylor(exp(x),x,0,10)
[p.coefficient(x,n)*factorial(n) for n in (0..10)]
\end{lstlisting}
\vspace{-2.5mm}
\begin{lstlisting}[basicstyle={\footnotesize\ttfamily},frame=single,backgroundcolor=\color{lightgray!20},title=Mathematica:]
CoefficientList[Series[E^x,{x,0,10}],x]*Range[0,10]!
\end{lstlisting}

We are going to use the following fundamental EGF results.

\begin{thr}
\cite[Theorem 8.21]{Bona} $($Product formula$)$ Let $a_{n}$ be
the number of ways to build a certain structure on an $n$-element
set, and let $b_{n}$ be the number of way to build another structure
on an $n$-element set. Let $c_{n}$ be the number of ways to split
$\{1,\ldots,n\}$ into the disjoint union of the subsets $S$ and
$T$, and then build a structure of the first kind on $S$, and a
structure of the second kind on $T$. If $A(x)$, $B(x)$, and $C(x)$
are the respective EGFs of the sequences
$(a_{n})$, $(b_{n})$, and $(c_{n})$, then $C(x)=A(x)B(x)$.
\end{thr}

\begin{thr}
\cite[Theorem 8.25]{Bona} $($Exponential formula$)$ Let $a_{n}$
be the number of ways to build a certain structure on an $n$-element
set, satisfying $a_{0}=0$. Let $h_{n}$ be the number of ways to
partition the set $\{1,\ldots,n\}$ into an unspecified number of
non-empty subsets, and then build a structure of the given kind on
each of these subsets. Set $h_{0}=1$. If $A(x)$ and $H(x)$ are
the respective EGFs of sequences $(a_{n})$
and $(h_{n})$, then $H(x)=e^{A(t)}$. 
\end{thr}

We also need a few more facts. Recall that the \emph{Stirling number}
$S_{n,k}$ of the second kind counts the number of ways to partition
a set with $n$ elements into $k$ nonempty subsets.

\begin{prop}\label{stirlingprop}
\cite[Example 8.23]{Bona} The EGF of
the Stirling numbers $(S_{n,k})_{n\ge0}$ is $S_{k}(x)=\frac{(e^{x}-1)^{k}}{k!}$.
\end{prop}

\begin{prop}
\cite[Example 8.24]{Bona} The EGF of
the Bell numbers is $B(x)=e^{e^{x}-1}$.
\end{prop}

The EGF of $(e_{n})$ will involve the so-called modified Bessel functions, which are solutions to certain differential equations. 
We will not give a precise definition here, but see for instance \cite{abramowitz+stegun}.

The following is a slight modification of a formula of \cite{bessel}. 
See also \cite[\href{https://oeis.org/A000984}{A000984}]{oeis}.

\begin{prop}
The modified Bessel function of the first kind $I_{0}$ satisfies
the identity $\sum_{n=0}^{\infty}{2n \choose n}\frac{x^{2n}}{(2n)!}=I_{0}(2x)$.
\end{prop}

Let $\tilde{e}_{n}$ be the number of freely and evenly tagged polydiagonal
subspaces of $\mathbb{R}^{n}$.

\begin{prop}
The EGF of $(\tilde{e}_{n})$ is $\tilde{E}(x)=e^{I_{0}(2x)/2-1/2}$.
\end{prop}

\begin{proof}
To build a freely and evenly tagged polydiagonal subspace $\Delta_{\mathcal{P}}$,
we partition $C=\{1,\ldots,n\}$ into an unspecified number of nonempty
subsets and then further divide each subset into two nonempty classes
of equal sizes. These two classes are mapped to each other by the
involution. The number of ways this further division can be done is
\[
a_{n}=\begin{cases}
\frac{1}{2}{n \choose n/2}, & n\text{ is positive and even}\\
0, & n\text{ is odd}.
\end{cases}
\]
The EGF of ($a_{n})$ is $I_{0}(2x)/2-1/2$.
The result now follows from the Exponential formula.
\end{proof}

\begin{prop}
\label{egfE}
The EGF of $(e_{n})$ is $E(x)=e^{I_{0}(2x)/2-1/2+x}$. 
\end{prop}

\begin{proof}
To build an evenly tagged polydiagonal subspace $\Delta_{\mathcal{P}}$,
we split $C=\{1,\ldots,n\}$ into the disjoint union of two subsets
and then we create a freely and evenly tagged polydiagonal subspace
on one of these subsets and use the other subset to create the fixed
point $P_{0}$ of the involution. The creation of the fixed point
set can be done in $u_{n}=1$ way. The EGF
of $(u_{n})$ is $U(x)=e^{x}$. The Product formula now implies that
$E(x)=\tilde{E}(x)U(x)=e^{I_{0}(2x)/2-1/2+x}$.
\end{proof}
Let $\tilde{f}_{n}$ be the number of freely and fully tagged polydiagonal
subspaces of $\mathbb{R}^{n}$.

\begin{prop}
The EGF of $(\tilde{f}_{n})$ is $\tilde{F}(x)=e^{(e^{2x}-2e^{x}+1)/2}$.
\end{prop}

\begin{proof}
To build a freely and fully tagged polydiagonal subspace $\Delta_{\mathcal{P}}$,
we partition $C=\{1,\ldots,n\}$ into an unspecified number of nonempty
subsets and then further divide each subset into two nonempty classes.
These two classes are mapped to each other by the involution. The
number of ways this further division can be done is the Stirling number
$S_{n,2}=2^{n-1}-1$. The Exponential formula, using Proposition~\ref{stirlingprop}, now implies that $\tilde{F}(x)=e^{S_{2}(x)}=e^{(e^{2x}-2e^{x}+1)/2}$.
\end{proof}

The sequence $(\tilde{f}_{n})$ is listed as \cite[\href{https://oeis.org/A060311}{A060311}]{oeis}.

\begin{prop}
\label{egfF}
The EGF of $(f_{n})$ is $F(x)=e^{(e^{2x}-2e^{x}+2x+1)/2}$.
\end{prop}

\begin{proof}
To build a fully tagged polydiagonal subspace $\Delta_{\mathcal{P}}$,
we split $C=\{1,\ldots,n\}$ into the disjoint union of two subsets
and then we create a freely and fully tagged polydiagonal subspace
on one of these subsets and use the other subset to create the fixed
point $P_{0}$ of the involution. The Product formula now implies
that $F(x)=\tilde{F}(x)e^{x}=e^{(e^{2x}-2e^{x}+2x+1)/2}$.
\end{proof}

The EGF of $(f_{n})$ coincides with the one in \cite[\href{https://oeis.org/A307389}{A307389}]{oeis}, proving that both sequences are indeed the same. 
This latter sequence counts the number of elements in the so-called species of orbit polytopes in dimension $n$, see \cite{hopfmonoid}.

\begin{prop}
\label{egfP}
The EGF of $(p_{n})$ is $P(x)=e^{(e^{2x}-1+2x)/2}$.
\end{prop}

\begin{proof}
To build a polydiagonal subspace $\Delta_{\mathcal{P}}$,
we split $C=\{1,\ldots,n\}$ into the disjoint union of two subsets
and then we further partition the first subset into classes and create
a fully tagged polydiagonal subspace on the second subset.
The number of ways the first subset can be partitioned is the Bell
number $B_{n}$. The Product formula now implies that $P(x)=B(x)F(x)=e^{(e^{2x}-1+2x)/2}$.
\end{proof}

This provides an alternate proof for the fact that the $p_n$ are the Dowling numbers.

\subsection{Recursion formulas}
\label{recursion}
The following recurrence relations can be used to find the numbers of some types of polydiagonal 
subspaces.  The initial values are $p_0 = s_0 = f_0 = e_0 = 1$ as discussed in Example~\ref{R0}.
Some of these recurrence relations are already known \cite{oeis}, but we include them for completeness.

Recall that the multinomial coefficient ${ n \choose k_1,k_2,\ldots,k_r}$ is the number of ways to label $n$ elements of a set with $r$ labels 
such that the $i$-th label is used exactly $k_i$ times. 
If $k_1+\cdots+k_r=n$ and $k_i$ is nonnegative for all $i$, then we can use the alternate notation
$$
{ n \choose k_1\ k_2\ \ldots\ k_{r-1}}:={ n \choose k_1,k_2,\ldots,k_r}=\frac{n!}{k_1! k_2!\cdots k_r!}.
$$
This is the binomial coefficient when $r = 2$.  

\begin{prop}
\label{recurrence}
If $n \geq 0$, then 
\begin{itemize}
\item
$\displaystyle
p_{n+1} = p_n +  \sum_{ k+\ell < n} {n \choose k \ \ell}  p_{k} +  \sum_{k=0}^n {n \choose k} p_{k}
$
\item
$\displaystyle
s_{n+1} = 
\sum_{k=0}^n {n \choose k} s_{k}
$
\item
$\displaystyle
f_{n+1} = 
f_n  +   \sum_{k+\ell< n} {n \choose k \ \ell} f_{k}
$
\item
$\displaystyle
e_{n+1} = 
e_n  + \sum_{k+2\ell+1 = n} {n \choose k \ \ell} e_{k}.
$
\end{itemize}

\end{prop}

\begin{proof}
Let $T_n$ be the set of tagged partitions of $C_n := \{1, \dots, n\}$.  We derive the recurrence relation for 
$p_{n+1}  = \# T_{n+1}$ by counting the size of
each set of the disjoint union $T_{n+1} = A \cupdot B \cupdot C$, where 
\[
\begin{aligned}
A &=\{\mathcal{P}\in T_{n+1} \mid [n+1]^* = [n+1]\} \\
B &=\{\mathcal{P}\in T_{n+1} \mid  [n+1]^* \neq [n+1]\}\\
C &=\{\mathcal{P}\in T_{n+1} \mid [n+1]^* \text{ is not defined} \}.
\end{aligned}
\]
The size of $A$ is $p_n$ since the removal of $n+1$ from $[n+1]$ is a bijection from $A$ to $T_n$.
This gives the first term in the recurrence relation.

To count $B$, we first count the elements in $B$ for which
$[n+1] \cap C_n$ has size $\ell$ and $[n+1]^*\cap C_n$ has size $m$, and then sum the counts.
We can partition $C_n$ into 3 blocks of size $k$, $\ell$ and $m$, with $k + \ell + m = n$, in  
$n \choose k \ \ell$ ways.  
Note that $m \geq 1$, so $k + \ell < n$. 
The classes and partial involutions on the $k$ 
vertices in $C_n \setminus([n+1] \cup [n+1]^*)$
can be chosen in $p_k$ ways.  
Thus, the size of $B$ is the second term in the recurrence relation for $p_{n+1}$

Next we count $C$ by summing the number of elements in $C$ for which $C_n \setminus [n+1]$ has size $0 \leq k \leq n$.
We can choose these $k$ elements from $C_n$ in $n \choose k$ ways.
There are $p_k$ ways to choose the classes and partial involutions on these $k$ elements.
Thus, the size of $C$ is the
third term in the recurrence relation for $p_{n+1}$.  

The other recurrence relations are derived by similar arguments.
\end{proof}

\begin{remk}
The sums in Proposition~\ref{recurrence} can be evaluated explicitly using
$$
\sum_{k+\ell < n} x_{k, \ell, n} = \sum_{k=0}^{n-1} \sum_{\ell = 0}^{n-k-1}x_{k, \ell, n}  \quad \text{and} \quad
\sum_{k+2\ell +1 = n} x_{k, \ell, n} = \sum_{\ell = 0}^{\lfloor \frac{n-1} 2 \rfloor} 
x_{n-2\ell-1,\ell, n}.
$$
\end{remk}

\section{Conclusion}
Polydiagonal spaces capture some of the most striking phenomena observed in network dynamical systems.
They correspond to multiple cells behaving in unison (synchrony) or opposition (anti-synchrony).
Polydiagonal subspaces are partitioned into synchrony subspaces and anti-synchrony subspaces.
Synchrony subspaces can be described in terms of equalities of the form $x_i = x_j$, meaning that cell $i$ is
synchronized with cell $j$.
Anti-synchrony subspaces have at least one equality of the form $x_i = - x_j$, including $x_i = 0$.

We have focused on linearly coupled cell networks of System~(\ref{ODE}), but there are many results stating
that a polydiagonal subspace is invariant for nonlinear networks with the coupling of 
Equation~(\ref{weight-additive-coupled})
if the subspace is $M$-invariant.
The matrix $M$ is the weighted adjacency matrix, or the weighted Laplacian matrix, 
of the network. 

Proposition~\ref{dynamicallyInvariant} describes the consequences of $M$-invariance for coupled cell networks. 
An $M$-invariant synchrony subspace is dynamically invariant
for System~(\ref{ODE}) with any internal dynamics function $f$.
An $M$-invariant anti-synchrony subspace is dynamically invariant for System~(\ref{ODE}) with any \emph{odd} internal dynamics function $f$.

We define several different types of anti-synchrony subspaces shown in Figure~\ref{classification}, and
identify which types can be $M$-invariant for several classes of coupled cell networks.
We focus on evenly tagged subspaces, where every cluster of synchronized cells 
has the same size as the corresponding cluster of anti-synchronized cells.
Evenly tagged subspaces
are precisely the polydiagonal subspaces which are orthogonal to the diagonal vector $\1$.
 
It is well-known that if each row sum of $M$ is the same, then the synchrony subspace $\spn(\1)$ is $M$-invariant.
For example, the Laplacian matrix of any weighted digraph has each row sum equal to 0,
so there is no coupling between cells of System~(\ref{ODEwithL}) if all the cells are synchronized.
We focus on the implications of the matrix $M$ having constant \emph{column} sums.  
Theorem~\ref{main} states that if $M$ has constant column sums, and certain non-degeneracy conditions hold,
then every $M$-invariant anti-synchrony subspace is evenly tagged.  
This has practical importance since the number of evenly tagged subspaces is much smaller than 
the number of all anti-synchrony subspaces.
This greatly simplifies the search for $M$-invariant subspaces if $M$ has constant column sums.
For example, the algorithm in \cite{NSS7} which finds $M$-invariant synchrony subspaces 
can be generalized to find $M$-invariant evenly tagged subspaces efficiently.

We find conditions on a weighted digraph so that the adjacency matrix, or the Laplacian matrix,
has constant column sums. Theorem~\ref{input-output} says that the Laplacian matrix of a weight-balanced digraph also has constant column sums.
Lemma~\ref{symimplieseul} says that 
a vertex transitive weighted digraph is weight-balanced.
Furthermore, the Laplacian matrix of a graph (as opposed to a weighted digraph)
is weight-balanced.
In summary, we have identified several classes of coupled cell networks where 
all of the dynamically invariant subspaces are either 
synchrony subspaces or evenly tagged anti-synchrony subspaces.

\section*{Data Availability and Conflict of Interest}
This paper's data is computer-generated  and available at \cite{ourGitHub}.  
The authors declare no financial conflict of interest.  
E.N. acknowledges the support of 
the Center for Research in Mathematics Applied to Industry (FAPESP Cemeai grant 2013/07375-0)  and the Serrapilheira Institute (Grant No. Serra-1709-16124)

\bibliographystyle{plain}
\bibliography{anti}

\end{document}